\documentclass[12pt,leqno]{amsart}
\usepackage{amsmath, amssymb, amscd, amsfonts, xypic,   stmaryrd, turnstile, mathrsfs, eucal,color}

\definecolor{dullmagenta}{rgb}{0.4,0,0.4}

\definecolor{darkblue}{rgb}{0,0,0.4}

\usepackage[colorlinks=true, pdfstartview=FitV, linkcolor=red, citecolor=blue, urlcolor=darkblue]
{hyperref}

\newtheorem{theorem}{Theorem}[section]
\newtheorem{lemma}[theorem]{Lemma}
\newtheorem{proposition}[theorem]{Proposition}
\newtheorem{corollary}[theorem]{Corollary}

\theoremstyle{definition}
\newtheorem{definition}[theorem]{Definition}
\newtheorem{example}[theorem]{Example}
\newtheorem{remark}[theorem]{Remark}

\begin{document}

\title[Loewy ]{The Loewy series of an FCP (distributive) ring extension}

\author[G. Picavet and M. Picavet]{Gabriel Picavet and Martine Picavet-L'Hermitte}
\address{Math\'ematiques \\
8 Rue du Forez, 63670 - Le Cendre\\
 France}
\email{picavet.mathu (at) orange.fr}

\begin{abstract} If $R\subseteq  S$ is a ring extension of commutative rings,  we consider the lattice $([R,S],\subseteq)$ of all  the $R$-subalgebras of $S$. We assume that the poset $[R,S]$ is both Artinian and Noetherian; that is, $R\subseteq S$ is an FCP extension.    The Loewy series of such lattices are studied. Most of  main results are
gotten in case these posets are distributive, which occurs  for integrally closed extensions. In general, the situation is much more complicated. We give a discussion for finite field extensions. 

\end{abstract}


\keywords  {FIP, FCP extension, minimal extension,   support of a module, distributive lattice, Boolean lattice, 
 atom,
  socle, Loewy series, Galois extension}

\maketitle

\section{Introduction and Notation}

  If $L$ is a complete  lattice, with smallest  and greatest elements, its socle 
$\mathcal S(L)$ is defined as the   supremum of all its atoms. Then the Loewy series  of $L$ is defined by transfinite  induction, where in particular $\mathcal S_{i+1}(L) = \mathcal S(\mathcal S_i(L))$
 for a positive integer $i$ (See Section 3 for more details).   When $L$ is the lattice of submodules of a module, the Loewy series  of $L$ is a well known topic and its theory a long chapter of algebra, even when the base ring is non-commutative. 

In this paper, we  consider the category of commutative and unital rings, whose    epimorphisms will be involved. If  $R\subseteq S$ is a (ring) extension, we denote  by $[R,S]$ the set of all $R$-subalgebras of $S$ and
  set $]R,S[: =[R,S]\setminus \{R,S\}$ (with a similar definition for $[R,S[$ or $]R,S]$). 
 
  We will consider lattices of the following form.
    For an extension $R\subseteq S$, the poset $([R,S],\subseteq)$ is a {\it complete} lattice,  where the supremum of any non void subset  is the compositum of its elements, which we call {\it product} from now on and denote by $\Pi$ when necessary, and the infimum of any non void subset is the intersection of its elements.    We emphasize on the following.  If $R\subseteq S$ is a ring extension, our main interest is   in the properties of the  Loewy series related to the lattice $[R,S]$, and not in the lattice of  $R$-submodules ($R/C$-submodules)  of $S/R$, where $C$ is the conductor of $R\subseteq S$, although there are some relations. Moreover,  we only consider extensions of finite length, in a sense defined  below, so that Loewy lengths are finite in this paper.    
      
    As a general rule, an extension $R\subseteq S$ is said to have some property of lattices if $[R,S]$ has this property.
 We use lattice definitions and properties described in \cite{NO}. 
 
 The extension $R\subseteq S$ is said to have FIP (for the ``finitely many intermediate algebras property") or  is
 an FIP  extension if $[R,S]$ is finite. A {\it chain} of $R$-subalgebras of $S$ is a set of elements of $[R,S]$ that are pairwise comparable with respect to inclusion.    
 We will say that $R\subseteq S$ is chained if $[R,S]$ is a chain. We  also say that the extension $R\subseteq S$ has FCP  (or is an FCP   extension) if each chain in $[R,S]$ is finite. Clearly,  each extension that satisfies FIP must also satisfy FCP. 
Dobbs and the authors characterized FCP and FIP extensions \cite{DPP2}. 
  
This paper is a continuation of our earlier  paper \cite{Pic 10}, where we considered Boolean ring extensions. It is devoted to the study of Loewy series of an FCP 
 (distributive)
 extension, a notion linked to Boolean extensions.  As much  as  possible, we give results for FCP extensions that are not necessarily distributive, in particular, for the behavior of the Loewy series with respect to classical constructions of ring theory. It may be asked whether the distributivity property may be replaced with the modular condition, since the lattice of submodules of a module is evidently modular.  
 
  In a  forthcoming paper, we study distributive extensions.
Note that integrally closed FCP extensions are distributive.

Our main tool will be  the minimal (ring) extensions, a concept that was introduced by Ferrand-Olivier \cite{FO}. 
In our context, minimal extensions coincide with atoms. They are completely known (see Section 2).  Recall that an extension $R\subset S$ is called {\it minimal} if $[R, S]=\{R,S\}$. 
  The key connection between the above ideas is that if $R\subseteq S$ has FCP, then any maximal (necessarily finite) chain $\mathcal C$ of $R$-subalgebras of $S$, $R=R_0\subset R_1\subset\cdots\subset R_{n-1}\subset R_n=S$, with {\it length} $\ell(\mathcal C):=n <\infty$, results from juxtaposing $n$ minimal extensions $R_i\subset R_{i+1},\ 0\leq i\leq n-1$. 
An FCP extension is finitely generated, and  (module) finite if integral.
For any extension $R\subseteq S$, the {\it length} $\ell[R,S]$ of $[R,S]$ is the supremum of the lengths of chains of $R$-subalgebras of $S$. Notice  that if $R\subseteq S$ has FCP, then there {\it does} exist some maximal chain of $R$-subalgebras of $S$ with length $\ell[R,S]$ \cite[Theorem 4.11]
{DPP3}.

Any undefined  material
 is explained at the end of the section or in the next sections. 

 Section 2 is  devoted  to some recalls and results on ring extensions  and their lattice properties.
 
 In Section 3, we study the Loewy series of a 
 arbitrary    FCP extension. 
  As a first property, the Loewy series behaves well with respect to localization (Proposition  \ref{8.8}).
  Let $R\subset S$ be a distributive FCP extension. Proposition \ref{8.3} shows that the {\it Loewy series} $S_0:=R\subset \ldots\subset S_i\subset \ldots\subset S_n:=S$ is such that $S_i\subset S_{i+1}$ is a Boolean extension for each $i=0,\ldots,n-1$. In particular, Theorem \ref{8.6} gives a characterization of such extensions verifying $[R,S]=\cup_{i=0}^n[S_i,S_{i+1}]$.   We give computations of Loewy series for some special extensions or some subextensions, for example for Nagata extensions. We  also show how to compute the Loewy series of some modules by using the Loewy series of a  ring extension.
 We give many examples. 
 For instance, if $R\subseteq S$ is an FCP almost-Pr\"ufer extension with Pr\"ufer hull $\tilde R$, the Loewy series of $R\subseteq S$ is gotten by using the Loewy series of $R\subseteq \overline R$ and $R\subseteq \tilde R$ (Corollary  \ref{8.13}).  
 Note here that ring extensions  whose Loewy length is  $1$ are, for example, Boolean extensions and pointwise  minimal extensions  
  (Corollary  \ref{8.31} and Proposition  \ref{8.312}).
  
  Section 4  specially deals   with field extensions. 
   We begin with the characterization of the Loewy series of a finite field extension $k\subseteq L$ by means of the Loewy series of $k\subseteq T$ and $k\subseteq U$, where $T$ (resp. $U$) is the separable (resp. radicial) closure of $k\subseteq L$ (Proposition  \ref{9.3}). 
  Loewy series  of finite cyclic field extensions (they are necessarily distributive)  are 
   completely determined 
  in Theorem~\ref{8.131}.
 
 We denote by $(R:S)$ the conductor of $R\subseteq S$. The integral closure of $R$ in $S$ is denoted by $\overline R^S$ (or by $\overline R$ if no confusion can occur).  The characteristic of a field $k$ is denoted by $\mathrm{c}(k)$. A purely inseparable field extension is called 
{\it radicial} in this paper.  In particular, if $k\subset L$ is a radicial FIP field extension, then $[k,L]$ is a chain. 

  Finally,  $|X|$  is the cardinality of a set $X$,   $\subset$ denotes proper inclusion and, for a positive integer $n$, we set $\mathbb{N}_n:=\{1,\ldots,n\}$.  

 \section {Recalls and results on ring extensions}
This section is devoted to two types of recalls: commutative rings and lattices.

\subsection{Rings and ring extensions}
  A {\it local} ring is here what is called elsewhere a quasi-local ring. As usual, Spec$(R)$ and Max$(R)$ are the set of prime and maximal ideals of a ring $R$.  
  The support of an $R$-module $E$ is $\mathrm{Supp}_R(E):=\{P\in\mathrm{Spec }(R)\mid E_P\neq 0\}$, and $\mathrm{MSupp}_R(E):=\mathrm{Supp}_R(E)\cap\mathrm{Max}(R)$. If $E$ is an $R$-module, ${\mathrm L}_R(E)$ (also denoted $L(M)$) is its length.

 If $R\subseteq S$ is a ring extension and $P\in\mathrm{Spec}(R)$, then $S_P$ is both the localization $S_{R\setminus P}$ as a ring and the localization at $P$ of the $R$-module $S$. 
  We denote by $\kappa_R(P)$ the residual field $R_P/PR_P$ at $P$. 
   An extension $R\subset S$ is called {\it locally minimal} if $R_P\subset S_P$ is minimal for each $P\in\mathrm{Supp}(S/R)$  or equivalently for each $P\in\mathrm{MSupp}(S/R)$.    

The following notions and results are  deeply involved in the sequel. 

\begin{definition}\label{crucial 1}\cite[Definition 2.10]{CPP} An extension $R\subset S$ is called  {\it $M$-crucial} if 
 $\mathrm{Supp}(S/R)=\{M\}$. Such $M$
is  called the {\it crucial (maximal) ideal}  $\mathcal{C}(R,S)$ of $R\subset S$. 
\end{definition}

\begin{theorem}\label{crucial}\cite[Th\'eor\`eme 2.2]{FO} A  minimal extension  is crucial  and is either integral ((module)-finite) or a flat epimorphism.
\end{theorem} 

  Recall  that an  extension $R\subseteq S$ is called {\it Pr\"ufer}  if $R\subseteq T$ is a flat epimorphism for each $T\in[R,S]$ (or equivalent, if $R\subseteq S$ is a normal pair) \cite[Theorem 5.2]{KZ}.
 In \cite{Pic 5}, we called an extension which is a minimal flat epimorphism, a {\it Pr\"ufer minimal} extension. Three types of minimal integral extensions exist, characterized in the next theorem, (a consequence of  the fundamental lemma of Ferrand-Olivier), so that there are four types of minimal extensions, mutually exclusive.

\begin{theorem}\label{minimal} \cite [Theorem 2.2]{DPP2} Let $R\subset T$ be an extension and  $M:=(R: T)$. Then $R\subset T$ is minimal and finite if and only if $M\in\mathrm{Max}(R)$ and one of the following three conditions holds:

\noindent (a) {\bf inert case}: $M\in\mathrm{Max}(T)$ and $R/M\to T/M$is a minimal field extension.

\noindent (b) {\bf decomposed case}: There exist $M_1,M_2\in\mathrm{Max}(T)$ such that $M= M _1\cap M_2$ and the natural maps $R/M\to T/M_1$ and $R/M\to T/M_2$ are both isomorphisms.

\noindent (c) {\bf ramified case}: There exists $M'\in\mathrm{Max}(T)$ such that ${M'}^2 \subseteq M\subset M',\  [T/M:R/M]=2$, and the natural map $R/M\to T/M'$ is an isomorphism.

In each of the above  cases, $M=\mathcal{C}(R,T)$.
\end{theorem}

\subsection{Lattice Properties}
 Let $R\subseteq S$ be an FCP extension, then $[R,S]$ is a complete Noetherian Artinian lattice, $R$ being the least element and $S$  the largest. In the context of the lattice $[R,S]$, some  definitions and properties of lattices have the following formulations. (see \cite{NO})

An element $T\in[R,S]$ is called 

(1) $\Pi$-{\it irreducible} (resp. $\cap$-{\it irreducible})  if $T=T_1T_2$ (resp.  $T=T_1\cap T_2$) implies $T=T _1$ or $T=T_2$. 

 (2) an {\it atom} (resp. a {\it co-atom})  if and only if  $R\subset T$   (resp. $T\subset S$)  is a minimal extension. 
  Therefore, an atom (resp. a co-atom)  is $\Pi$-irreducible (resp. $\cap$-irreducible). We denote by $\mathcal{A}$ (resp. $\mathcal{CA}$) the set of atoms (resp. co-atoms) of $[R,S]$. 
    Theorems \ref{crucial} and \ref{minimal} show that there are four types of atoms.
 
 (3)  {\it essential} if $T\neq R$ and $ U\cap T \neq  R$ for each $U\in ]R,S]$. If $N$ is an $R$-submodule of an $R$-module $M$, then $N$ is {\it essential} as a submodule if $N\cap N'\neq 0$ for any submodule $N'\neq 0$ of $M$. Clearly, $T$ is essential if  $T/R$ is an essential submodule of the $R$-module $S/R$.  

(4)  $R\subseteq S$ is called {\it catenarian}, or graded by some authors,  if $R\subset S$ has FCP and all maximal chains between two comparable elements have the same length.

(5) $R\subseteq S$ is called {\it distributive} if intersection and product are each distributive with respect to the other. Actually, each distributivity implies the other \cite[Exercise 5, page 33]{NO}. 

(6) Let $T\in[R,S]$. Then, $T'\in[R,S]$ is called a {\it complement} of $T$ if $T\cap T'=R$ and $TT'=S$. 

(7) An extension $R\subseteq S$ is called  {\it Boolean} if 
 $([R,S],\cap,\cdot)$ is a  distributive lattice such that each $T\in[R,S]$ has a (necessarily unique) complement.

   \begin{proposition} \label{1.0} \cite[Theorem 1, p. 172]{G} A distributive lattice of finite length is catenarian (the Jordan-H\"older chain condition holds).
\end{proposition} 

\begin{proposition}\label{5.51}  If $R\subseteq S$ has FCP, then  $T\in [R,S]$ is $\cap$-irreducible (resp$.$ $\Pi$-irreducible) if and only if either $T=S$ (resp$.$ $T=R$) or there is a unique $T' \in[R,S]$ such that $T\subset T'$ (resp$.$ $T'\subset T$) is  minimal.  
\end{proposition}

\begin{proof} Obvious.
\end{proof} 

\begin{definition}\label{4.2} A ring extension $R\subseteq S$ is called 
{\it arithmetic} if $[R_P, S_P]$ is a chain for each $P\in\mathrm{Spec}(R)$.

An arithmetic extension is distributive by \cite[Proposition 5.18]{Pic 4}.
\end{definition}

In the next proposition, we need the following definition: 
A ring extension $R\subset S$ is called {\it quadratic} if each $t\in S$ is a zero of a monic quadratic polynomial over $R$ (\cite[definition page 430]{HP}). 

\begin{proposition}\label{4.3}  Let $R\subset S$ be a ring extension and let $T\in ]R,S]$. 

 \begin{enumerate}
 \item If $T/R$ is an essential $R$-submodule of $S/R$, then $T$ is an essential $R$-subalgebra of $S$.
 
 \item If  in addition, $R\subset S$ is quadratic, then $T$ is an essential $R$-subalgebra of $S$ if and only if $T/R$ is an essential $R$-submodule of $S/R$.  
  \end{enumerate}
\end{proposition}

\begin{proof} (1) Obvious.

(2) One part is (1). Assume that $R\subset S$ is quadratic and that  $T$ is an essential $R$-subalgebra of $S$. Let $N'$ be a nonzero $R$-submodule of $S/R$. There exists an $R$-submodule $N$ of $S$ containing $R$ such that $N'=N/R$. Let $t\in N\setminus\{R\}$. Then, $R\subset R[t]=R+Rt\subseteq N$ because $t$ satisfies a monic quadratic polynomial over $R$.  It follows that $R\neq T\cap R[t]\subseteq T\cap N$, which shows that $0\neq (T/R)\cap  (N/R)=(T/R)\cap N'$ and $T/R$ is an essential $R$-submodule of $S/R$.  
\end{proof} 

\begin{remark}\label{4.4} There exist ring extensions such that the equivalence of (2) in Proposition  \ref{4.3} does not hold. Let $k\subset L$ be a radicial FIP extension of degree $p^2$, where $\mathrm{c}(k)=p$. There exists a unique $K\in[k,L]$ such that $[K:k]=p$ because $[k,L]$ is a chain, so that $[k,L]=\{k,K,L\}$. Then, $K$ and $L$ are both essential $k$-subalgebras of $S$. But $K/k$ is not an essential $k$-vector subspace of $L/k$. Indeed, $[L:k]=p^2$ shows that there exists a basis $\{x_1,\ldots,x_{p^2}\}$ of the $k$-vector space $L$ such that $\{x_1,\ldots,x_p\}$ is a basis of the $k$-vector space $K$. Let $V:=k+kx_{p+1}$. Then, $(V/k)\cap (K/k)=0$ although $V/k\neq 0$.
\end{remark} 

   \section {The Loewy series of an FCP (distributive) extension}

We first note that a distributive  FCP extension  $R\subset S$  has FIP (\cite[Theorem 4.28]{R}). In this section, we  associate a chain in $[R,S]$ to the lattice $[R,S]$, called the  Loewy series of $[R,S]$.

\begin{definition}\label{8.1} \cite[Definitions pages 47, 51 and 77]{Cal}, \cite[page 29]{NO} Let $R\subset S$ be an FCP extension, $\mathcal A$ the set of atoms of $[R,S]$, $\mathcal {CA}$  the set of co-atoms and $\mathcal E$ the set of essential elements.

 \begin{enumerate}
 \item   The {\it socle} of the extension $R\subset S$ is  $\mathcal S[R,S]:=\prod_{A\in \mathcal A}A$. 
  
\item The {\it radical} of the extension $R\subset S$ is  $\mathcal R[R,S]:=\cap_{A\in \mathcal {CA}}A$.
 
 \item The {\it Loewy series} of the extension $R\subset S$ is the chain $\{S_i\}_{i=0}^n$ defined by induction as follows: $S_0:=R,\ S_1:=\mathcal S[R,S]$ and for each $i\geq 0$ such that $S_i\neq S$, we set $S_{i+1}:=\mathcal S[S_i,S]$.
 
  \item The {\it Loewy length} $\pounds[R,S]$ of the extension $R\subset S$ is the least integer $n$ such that $S=S_n$. Of course, since $R\subset S$ has FCP, there is some integer $n$ such that $S_n=S_{n+1}=S$. 
 \end{enumerate}
\end{definition}

 \begin{proposition} \label{8.22} Let $R\subset S$ be an FCP extension. Then, $\pounds[R,S]\leq \ell[R,S]\leq {\mathrm L}_R(S/R)$.
\end{proposition}

\begin{proof} Since the Loewy series $\{S_i\}_{i=0}^n$ of $R\subset S$ is a chain, which is not necessarily  maximal, we have $\ell[R,S]\geq\sum_{i=0}^{n-1}\ell[S_i,S{i+1}]\geq n=\pounds[R,S]$. Now, it exists $\{R_j\}_{i=0}^m$, a maximal chain such that $m=\ell[R,S]$  \cite[Theorem 4.11]
{DPP3}. Then, ${\mathrm L}_R(S/R)=\sum_{j=0}^{m-1}{\mathrm L}_R[R_{j+1}/R_j]$ (\cite[Theorem 6, page 20]{N}), with ${\mathrm L}_R[R_{j+1}/R_j]\geq 1$ for each $j\in\{0,\ldots,m\}$, so that ${\mathrm L}_R(S/R)\geq \ell[R,S]$. 
\end{proof} 

\begin{proposition} \label{8.111} Let $R\subset S$ be an FCP extension. Then, $\mathcal S[R,S]=\cap_{E\in\mathcal E}E$ and is the least element of $\mathcal E$.
\end{proposition}

\begin{proof} By definition, $\mathcal S[R,S]:=\prod_{A\in \mathcal A}A$. Let $T\in  \mathcal E$ and  
$A\in \mathcal A$. Then, $R\subset T\cap A\subseteq A$ shows that $T\cap A=A$  since $R\subset A$ is minimal,  which leads to $A\subseteq T$, so that any element of $\mathcal A$ is contained in any element of $\mathcal E$. In particular, $\mathcal S[R,S]\subseteq\cap_{E\in\mathcal E}E\ (*)$.  Let $U\in]R,S]$. Since $R\subset S$ has FCP, there exists some $A\in \mathcal A$ such that $A\subseteq U$, so that $A\cap U=A$. But $A\subseteq \mathcal S[R,S]$ shows that $R\subset A\subseteq U\cap \mathcal S[R,S]$. Then, $\mathcal S[R,S]$ is essential and $(*)$ gives the result.
\end{proof} 

We recall that  an $R$-module $N\neq 0$ is called {\it simple} if it is an atom in the lattice $\Lambda(N)$ of $R$-submodules of $N$, which is equivalent to $N=Rx$ for any nonzero $x\in N$ 
 \cite[Lemma 2.4.1]{NO}.
 In the following corollary, we  set $\mathcal {MS}[R,S]:=\sum\{N\in\Lambda(S)\mid R\subset N,\ N/R$ simple in $\Lambda (S/R) \}$. ($\mathcal {MS}[R,S]$ stands for module-socle). In view of \cite[page 52]{NO}, $\mathcal {MS}[R,S]$ is the intersection of the $R$-submodules $N$ of $S$ containing $R$ such that $N/R$ is essential in $S/R$.

\begin{corollary} \label{8.112} Let $R\subset S$ be an integral $M$-crucial FCP extension. Then  $\mathcal S[R,S]\subseteq\mathcal {MS}[R,S]$ with equality if $R\subset S$ is quadratic.
 \end{corollary}

\begin{proof} Set $k:=R/M$. Let $A\in\mathcal A$. Since $R\subset A$ is minimal, we have $M=(R:A)\in \mathrm{Max}(R)$, so that $M(\mathcal S[R,S])\subseteq R$. Let $x\in\mathcal S[R,S]$ and set $N:=R+Rx\subseteq R[x]$, which gives $MN\subseteq R$. Let $\overline x$ be the class of $x$ in $N/R$. Then, $N/R=R\overline x$. But $N/R$ is also a one-dimensional $k$-vector space generated by $\overline x$.  It follows by \cite[Corollary 2, page 66]{N} that ${\mathrm L}_R(N/R)={\mathrm L}_k(N/R)=\dim_k(N/R)=1$ showing that $N/R$ is a simple $R$-submodule of $S/R$. Then, $x\in \mathcal {MS}[R,S]$, which leads to $\mathcal S[R,S]\subseteq \mathcal {MS}[R,S]$.

Assume in addition that $R\subset S$ is  quadratic. Let $x\in S$ be such that $N:=R+Rx$ satisfies the following: $N/R$ is a simple $R$-submodule of $S/R$. Since $R\subset S$ is quadratic, it follows that $N\in[R,S]$. Moreover, the fact that $N/R$ is a simple $R$-submodule of $S/R$ shows that there is no $R$-submodule, and a fortiori, no $R$-subalgebra of $S$ strictly contained between $R$ and $N$, so that $N\in\mathcal A\subseteq \mathcal S[R,S]$. This property holding for any $R$-submodule $N$ of $S$ containing $R$ such that $N/R$ is a simple $R$-submodule, we get that $ \mathcal {MS}[R,S]
\subseteq \mathcal S[R,S]$, with  equality because of the first part.
\end{proof} 

 \begin{remark}\label{8.113} We use the example of Remark  \ref{4.4} to show that, in general, the two socles of Corollary  \ref{8.112} do not coincide.   Let $k\subset L$ be a radicial FIP extension of degree $p^2$, where $\mathrm{c}(k)=p$. There exists a unique $K\in[k,L]$ such that $[K:k]=p$, so that $K$ is the unique atom of $[k,L]$, giving that $K=\mathcal S[k,L]$.  But $[L:k]=p^2$ shows that there exists a basis $\{x_1:=1,\ldots,x_{p^2}\}$ of the $k$-vector space $L$, with $L=\sum_{i=1}^{p^2}kx_i$. For $i>1$, set $V_i:=k+kx_i$. Then, each $V_i/k$ is a simple subspace of $L/k$ and $L=\mathcal {MS}[k,L]\neq \mathcal {S}[R,S]=K$.
\end{remark} 
 
\begin{lemma} \label{8.2} Let $R\subset S$ be a distributive FCP extension (hence FIP) and let $S_1$ be its socle. Then, $S_1=\sup\{T\in[R,S]\mid R\subset T$  Boolean$\}$. 
\end{lemma}

\begin{proof} $R\subset S_1$ is Boolean because distributive and  $S_1$ is a product of atoms of $R\subset S_1$ (\cite[page 292]{St}). Let $T\in[R,S]$ be such that $R\subset T$ is Boolean. It follows from 
 \cite[Theorem 3.1]{Pic 10} that
$T=\prod_{A\in \mathcal B}A$, where   $\mathcal B\subseteq \mathcal A$, and then  $T\subseteq S_1$.
\end{proof}

\begin{corollary} \label{8.21} Let $R\subset S$ be a distributive FCP extension (hence FIP).  Then $\mathcal S[R,T]=\mathcal S[R,S]\cap T$ holds for each $T\in ]R,S]$. 
\end{corollary}

\begin{proof} Set $S_1:=\mathcal S[R,S],\ T_1:=\mathcal S[R,T]$ and 
$T'_1:=S_1\cap T$. Because $T\neq R$, there exists some $A\in\mathcal A$ such that $R\subset A\subseteq T$. So, $R\subset A\subseteq S_1\cap T=T'_1\subseteq S_1$. Since 
$R\subset S_1$ is Boolean, so is $R\subset T'_1$ by \cite[Proposition 3.11]{Pic 10}. Then, $T'_1\subseteq T_1$ in view of Lemma \ref{8.2}. But $R\subset T_1$ being also Boolean, it follows that $T_1\subseteq S_1$. From $T'_1\subseteq T_1\subseteq T\cap S_1=T'_1$ we infer that $T'_1=T_1$.
\end{proof} 

The Loewy series of a distributive FCP extension provides a  chain
of Boolean subextensions of this extension. 

\begin{proposition} \label{8.3} If  $R\subset S$ is a distributive FCP extension (hence FIP) and  $\{S_i\}_{i=0}^n$ 
 is
its Loewy series, then, $S_i\subset S_{i+1}$ is Boolean for each  $0\leq i \leq n-1$ whence, is either locally integral or locally Pr\"ufer.
Moreover, $\ell[R,S]=\sum_{i=0}^{n-1}\ell[S_i,S_{i+1}]$. 
\end{proposition}

\begin{proof}  By Definition \ref{8.1} (3), for each $i\in\{0,\ldots,n-1\}$, $S_{i+1}:=\mathcal S[S_i,S]$ holds, so that $S_i\subset S_{i+1}$ is Boolean by  Lemma \ref{8.2}. Since $R\subset S$ has FCP, the chain stops for some positive integer $n$ such that $S_n=S$. Deny, then there is some $S'$ which is an atom of $[S_n, S]$, a contradiction. 

  For each $i\in\{0,\ldots,n-1\},\ S_i\subset S_{i+1}$ is either locally integral or locally Pr\"ufer \cite[Proposition 3.4 and Corollary 3.19]{Pic 10}.

Since $R\subset S$ is distributive, $\ell[R,S]=\sum_{i=0}^{n-1}\ell[S_i,S_{i+1}]$ by Proposition~\ref{1.0}. 
\end{proof}

\begin{corollary} \label{8.31} Let  $R\subset S$ be a distributive FCP extension (hence FIP) and  $\{S_i\}_{i=0}^n$ 
its Loewy series. Then, 
 \begin{enumerate}
 \item $\pounds[R,S]=\ell[R,S]$ if and only if $[R,S]$ is a chain. In this case, $[R,S]=\{S_i\}_{i=0}^n$.
 
  \item $\pounds[R,S]=1$ if and only if $R\subset S$ is Boolean.
  \end{enumerate}
\end{corollary}

\begin{proof} Assume that $n=\ell[R,S]$. Proposition \ref{8.3} implies $n=\ell[R,S]\Leftrightarrow \ell[S_i,S_{i+1}]=1$ for each $i\in\{0,\ldots,n-1\}\Leftrightarrow S_i\subset S_{i+1}$ is minimal for each $i\in\{0,\ldots,n-1\}\Leftrightarrow S_i\subset S$ has only one atom (which is $S_{i+1}$) for each $i\in\{0,\ldots,n-1\}$. We show by induction on $i\in\{0,\ldots,n\}$ that if $T\in[R,S]$ is such that $\ell[R,T]=i$, then $T=S_i$. For $i=0$, obviously, $T=R=S_0$. Assume that the induction hypothesis holds for $i-1$; that is, $S_{i-1}$ is the only element $U$ of $[R,S]$ such that $\ell[R,U]=i-1$. Since $\ell[R,T]=i$ and $R\subset S$ is distributive and  FCP, any element $T'$ of $[R,S]$ such that $T'\subset T$ is minimal satisfies $\ell[R,T']=i-1$, so that $T'=S_{i-1}$ and $S_{i-1}\subset T$ is minimal, giving that $T$ is an atom of $[S_{i-1},S]$, that is $T=S_i$. The induction holds for any $i\in\{0,\ldots,n\}$, and $[R,S]=\{S_i\}_{i=0}^n$ is a chain. The converse is immediate using the fact that $S_i\subset S$ has only one atom (which is $S_{i+1}$).

Again, by Proposition \ref{8.3}, $n=1\Leftrightarrow S=S_1\Leftrightarrow R\subset S$ is Boolean.
\end{proof} 

\begin{example} \label{8.32} Let  $R\subset S$ be an  integrally closed FCP  extension such that $R$ is a local ring. Then \cite[Theorem 6.10]{DPP2} says that   $[R,S]$ is a chain,  therefore distributive \cite[Proposition 2.3]{Pic 10} and  
 Corollary  \ref{8.31} shows that $[R,S]=\{S_i\}_{i=0}^n$.
\end{example}  

 \begin{proposition} \label{8.311} Let $R\subset S$ be an  extension of length 2. The following conditions are equivalent: 
 \begin{enumerate}
 \item $[R,S]$ is a chain. 
 
\item  $|[R,S]|=3$.

\item $\pounds[R,S]=2$.
 \end{enumerate}
 If these conditions do not hold, then $\pounds[R,S]=1$.
\end{proposition} 
 \begin{proof} (1) $\Leftrightarrow$ (2) Obvious.
 
 (3) $\Leftrightarrow$ (1) by Corollary  \ref{8.31}.
 
 Assume that these conditions do not hold, then $\pounds[R,S]<\ell[R,S]$ by Proposition  \ref{8.22} leads to $\pounds[R,S]=1$.
\end{proof} 
 
  Recall that an extension $R\subset S$ is called {\it pointwise minimal} if $R\subset R[x]$ is minimal for each $x\in S\setminus R$. These extensions were studied by Cahen and the authors in \cite{CPP}.
 
 \begin{proposition} \label{8.312} If $R\subset S$ is  pointwise minimal, then,  $\pounds[R,S]=1$.
\end{proposition} 
 \begin{proof} 
 Any atom is a simple extension of $R$ (see \cite[Page 370, before Lemma 1.2]{Pic}). Conversely, let $x\in S\setminus R$, so that  $R\subset R[x]$ is minimal and $R[x]$ is an atom of $[R,S]$. Then, $S=\prod_{x\in S\setminus R}R[x]$ gives that $S=\mathcal S[R,S]=S_1$ and $\pounds[R,S]=1$.
 \end{proof} 
  In order to look at the behavior of  Loewy series under localizations, we need the following Lemma.

\begin{lemma} \label{8.321} Let $R\subset S$ be an FCP extension,  $T\subset U $ a subextension and $M\in\mathrm{Supp}(S/R)$. 

 \begin{enumerate}
 \item If $T\subset U$ is minimal, then either $T_M\subset U_M$ is minimal, or $T_M= U_M$. 
 
\item  If $T_M\subset U_M$ is minimal, there exists $V\in[T,S]$ such that $T\subset V$ is minimal with $V_M=U_M$.
 \end{enumerate}
 \end{lemma} 
 
 \begin{proof} (1) is \cite[Lemme 1.3]{FO}.
 
 (2) Assume that $T_M\subset U_M$ is minimal. Let $Q:=\mathcal C(T_M,U_M)\in\mathrm {MSupp}_{T_M}(U_M/T_M)\subseteq \mathrm{Max}(T_M)$. There exists $N\in\mathrm{Spec}(T)$, such that $Q=N_M$. 
 In particular, we have $(T_M)_Q\subset (U_M)_Q$  minimal $(*)$ and $(T_M)_P=(U_M)_P$ for each $P\in \mathrm{Spec}(T_M),\ P\neq Q\ (**)$. Since $(T_M)_Q=(T_M)_{N_M}=T_N$ and $(U_M)_Q=(U_M)_{N_M}=U_N$ by \cite[Proposition 7, page 85]{ALCO},  $(*)$ implies that $T_N\subset U_N$ is  minimal. Then, there exists some $V\in[T,S]$ such that $T\subset V$ is minimal with $V_N=U_N=(U_M)_Q=(V_M)_Q$  \cite[Lemma 3.3]{Pic 10}. It follows that  $V_{N'}=T_{N'}$ for any $N'\in \mathrm{Spec}(T),\ N'\neq N$. But $(**)$ gives, for $N'\in\mathrm{Spec}(T)$ such that $P=N'_M$ where $P\neq Q$, that $N'\neq N$, so that $(T_M)_P=(U_M)_P=U_{N'}=V_{N'}=(V_M)_P$, which leads to $U_M=V_M$. To conclude, if $T_M\subset U_M$ is minimal, there exists $V\in[T,S]$ such that $T\subset V$ is minimal with $V_M=U_M$.
 \end{proof} 

\begin{proposition} \label{8.8} Let $R\subset S$ be an FCP extension with Loewy series $\{S_i\}_{i=0}^n$. Then:
 \begin{enumerate}
 \item   Let $I$ be an  ideal shared by $R$ and $S$. Then $\{S_i/I\}_{i=0}^n$ is the Loewy series of $R/I\subset S/I$.
 
  \item Let $M\in\mathrm{Supp}(S/R)$ and $\{S'_i\}_{i=0}^{n_M}$ be the Loewy series of $[R_M,S_M]$. Then $S'_i=(S_i)_M$ for each $i\in\{0,\ldots,n_M\}$, where $n_M=\inf\{i\in\mathbb N_n\mid M\not\in\mathrm{Supp}(S/S_i)\}$.
  \end{enumerate}
\end{proposition} 

\begin{proof} (1) Let $T,U\in[R,S]$. By \cite[Corollary 1.4]{Pic}, $T\subset U$ is minimal if and only if $T/I\subset U/I$ is minimal. Then, an easy induction on $i$ shows that the Loewy series of $R/I\subset S/I$ is  $\{S_i/I\}_{i=0}^n$.

(2) Let $M\in\mathrm{Supp}(S/R)$. We show by induction on $i\in\{0,\ldots,n_M\}$ that $S'_i=(S_i)_M$.

The induction hypothesis is satisfied for $i=0$.

Assume that it holds for $i\in\{0,\ldots,n_M-1\}$, that is $S'_i=(S_i)_M$. Let $A$ be an atom of $[S_i,S]$, so that $S_i\subset A$ is minimal. Then, from Lemma  \ref{8.321}, we infer that either $(S_i)_M\subset A_M$ is minimal, or $(S_i)_M=A_M$, so that $A_M\in[S'_i,S'_{i+1}]$. It follows that $(S_{i+1})_M\subseteq S'_{i+1}$. 

 Now  $M\not\in\mathrm{Supp}(S/S_i)$ leads to $S'_i=(S_i)_M=S_M=S'_{i+1}$ and $i\geq n_M$, a contradiction,  so that 
 $M\in\mathrm{Supp}(S/S_i)$. Let 
 $B'$ be an atom of $[(S_i)_M,S_M]$, so that $(S_i)_M=S'_i\subset B'$ is minimal. There exists some $U\in[S_i,S]$ such that $B'=U_M$, with $(S_i)_M\subset U_M$  minimal. By Lemma  \ref{8.321}, there exists $B\in[S_i,S]$ such that $S_i\subset B$ is minimal, with $B_M=U_M=B'$, so that $B\subseteq S_{i+1}$, giving $B'=B_M\subseteq (S_{i+1})_M$.  Since $S'_{i+1}$ is the product of all atoms $B'$ of $[(S_i)_M,S_M]$, we get that $S'_{i+1}\subseteq (S_{i+1})_M$. To conclude, $S'_{i+1}= (S_{i+1})_M$ and the induction hypothesis holds for $i+1$. As we have just seen before, we get $S'_i=(S_i)_M=S_M=S'_{i+1}$ as soon as $M\not\in\mathrm{Supp}(S/S_i)$, so that $n_M=\inf\{i\in\mathbb N_n\mid M\not\in\mathrm{Supp}(S/S_i)\}$.
\end{proof} 

\begin{corollary} \label{8.82} Let $R\subset S$ be an FCP extension. Then, $\pounds[R,S]=\sup\{\pounds[R_M,S_M]\mid M\in\mathrm{MSupp}(S/R)\}$.
\end{corollary}

\begin{proof}  Let $M\in\mathrm{Supp}(S/R)$ and $\{S'_i\}_{i=0}^{n_M}$ be the Loewy series of $[R_M,S_M]$. Set $n:=\pounds[R,S]$. We proved in Proposition  \ref{8.8} that $S'_i=(S_i)_M$ for each $i\in\{0,\ldots,n_M\}$, where $n_M=\pounds[R_M,S_M]=\inf\{i\in\mathbb N_n\mid M\not\in\mathrm{Supp}(S/S_i)\}$. Then, for any $M\in\mathrm{Supp}(S/R),\ \pounds[R_M,S_M]\leq \pounds[R,S]$. Assume that for any $M\in\mathrm{Supp}(S/R),\ \pounds[R_M,S_M]< \pounds[R,S]=n$. It follows that $S'_{n-1}=(S_{n-1})_M=S_M$, and hence $S_{n-1}=S$, contradicting the definition of $n$. Then, $\pounds[R,S]=\sup\{\pounds[R_M,S_M]\mid M\in\mathrm{MSupp}(S/R)\}$.
\end{proof}  

\begin{example} \label{8.81}  
 We use the   example of Remark  \ref{4.4} in order to exhibit a computation of $n_M$ in Proposition \ref{8.8}:
 
  $k\subset L$ is a radicial FIP field extension of degree $p^2$ and $K$ is the only proper subalgebra of $L$. Set $R:=k^2,\ R_1:=[k[X]/(X^2)]\times k,\ R_2:=k\times K,\ R_3:=k\times L,\ S:=[k[X]/(X^2)]\times L,\ M:=0\times k$ and $N:=k\times 0$. Then, $\mathrm{Max}(R)=\{M,N\}$ with $M\neq N$.  
  Moreover, $k\subset k[X]/(X^2)$ is a minimal ramified extension and $S=R_1R_3$. By \cite[Lemma III.3 (d)]{DMPP}, $[R,S]=\{R,R_1,R_2,R_1R_2,R_3,S\}$.  We have the following diagram, where $R_1R_2=[k[X]/(X^2)]\times K$: 
 $$\begin{matrix}
 {} &        {}      & R_1 &       {}      &{}             & {}            & {} \\
 {} & \nearrow & {}     & \searrow & {}            & {}            & {} \\
 R &       {}      & {}     &      {}        & R_1R_2 & {}            & {} \\
 {} & \searrow & {}     & \nearrow & {}            & \searrow & {} \\
 {} &      {}       & R_2 & \to           & R_3        & \to           & S\end{matrix}$$
From \cite[Proposition III.4]{DMPP}, we deduce that  $R\subset R_1$  is a minimal extension with ${\mathcal C}(R,R_1)=M,\ R\subset R_2$ is a minimal  extension with ${\mathcal C}(R,R_2)=N\neq M$, so that  $T$ is an atom of $[R,S]$ if and only if $T\in\{k\times K, [k[X]/(X^2)]\times k\}=\{R_1,R_2\}$, and then $S_1=R_1R_2$. Using again \cite[Proposition III.4]{DMPP}, we get  that $S_1\subset S$ is minimal, so that $S_2=S$. Then, $n=2$. Now,  $(S_1)_M=(R_1)_M=S_M=(S_2)_M$, so that $ n_M=1$  and $(S_1)_N=(R_2)_N\neq  S_N=(R_3)_N=(S_2)_N$, whence  $n_N=2$.
 \end{example} 

  Let $R\subset S$ be an FCP extension and $\mathrm{MSupp}(S/R):=\{M_1,\ldots,M _n\} $. Consider the map $\varphi:[R,S]\to{\prod}_{i=1}^n[R_{M_i},S_{M_i}]$ defined by $\varphi(T):=(T_{M_1},\ldots,T_{M_n})$. Then $\varphi$ is injective  \cite[Theorem 3.6]{DPP2}. In \cite{Pic 10}, we called   $R\subseteq S $  a  {\it $\mathcal B$-extension} if $\varphi$ is bijective ($\mathcal B$ stands for bijective).  We proved in \cite[Proposition 2.20]{Pic 10} that  $R \subseteq S$ is a $\mathcal B$-extension if and only if $R/P$ is  local for each $P\in\mathrm{Supp}(S/R)$. This condition   holds in case $\mathrm {Supp}(S/R)\subseteq\mathrm {Max}(R)$, and, in particular, if $R\subset S$ is  integral. The Loewy series of such extensions have nice properties. 

\begin{proposition} \label{8.33} Let $R\subset S$ be an FCP $\mathcal B$-extension. For each $M\in\mathrm{MSupp}(S/R)$, set $ \mathcal A_M:=\{A\in \mathcal A\mid \mathcal{C}(R,A)=M\}$. Then:
\begin{enumerate}
 \item $\mathcal S[R_M,S_M]=\prod_{A\in\mathcal A_M} A_M$.
 
 \item There exists a (unique) $S_{1,M}\in[R,S]$ such that $(S_{1,M})_M=\mathcal S[R_M,S_M]$ and $(S_{1,M})_{M'}=R_{M'}$ for any $M'\in\mathrm{Spec}(R),\ M'\neq M$. In particular, $S_{1,M}=\prod_{A\in\mathcal A_M} A$.
 
 \item $\mathcal A=\cup_{M\in\mathrm{MSupp}(S/R)}\mathcal A_M$ and $S_1=\prod_{M\in\mathrm{MSupp}( S/R)}S_{1,M}$.
 \end{enumerate}
\end{proposition}

\begin{proof} (1) $(\mathcal S[R,S])_M=\mathcal S[R_M,S_M]$ by Proposition \ref{8.8}. It follows that $\mathcal S[R_M,S_M] =(\prod_{A\in\mathcal A} A)_M=\prod_{A\in\mathcal A} A_M$.  Let $M'\in\mathrm{MSupp}( S/R),\ M'\neq M$ and $A'\in\mathcal A_{M'}$. Then, $\mathcal{C}(R,A')=M'$ implies $A'_M=R_M$, so that  $\mathcal S[R_M,S_M]=\prod_{A\in\mathcal A_M} A_M$.

(2) Since $R\subset S$ is a $\mathcal B$-extension, the bijective map $\varphi:[R,S]\to {\prod}_{i=1}^n[R_{M_i},S_{M_i}]$ shows that there exists a unique $S_{1,M}\in[R,S]$ such that $(S_{1,M})_M=\mathcal S[R_M,S_M]=\prod_{A\in\mathcal A_M} A_M$ and $(S_{1,M})_{M'}=R_{M'}$ for any $M'\in\mathrm{Spec}(R),\  M'\neq M$. Moreover, for $A\in\mathcal A_M$, we have $A_{M'}=R_{M'}$, giving $\prod_{A\in\mathcal A_M} A_{M'}=R_{M'}$, so that $S_{1,M}=\prod_{A\in\mathcal A_M} A$.

(3) $\mathcal A=\cup_{M\in\mathrm{MSupp}(S/R)}\mathcal A_M$ because, for each $A\in\mathcal A$, we have $\mathcal{C}(R,A)\in\mathrm{MSupp}(S/R)$, which leads to $S_1=\prod_{A\in\mathcal A} A=\prod_{M\in\mathrm{MSupp}( S/R)}(\prod_{A\in\mathcal A_M} A)$

\noindent $=\prod_{M\in\mathrm{MSupp}( S/R)}S_{1,M}$. 
\end{proof} 

\begin{proposition} \label{8.34} Let $R\subset S$ be an integral arithmetic FCP extension with Loewy series $\{S_i\}_{i=0}^n$. For each $M\in\mathrm{MSupp}(S/R)$, let $\{R_{M,i}\}_{i=0}^{n_M}$ be the maximal chain of $[R_M,S_M]$,  set $ \mathcal A_M:=\{A\in \mathcal A\mid \mathcal{C}(R,A)=M\}$ and $m:=\sup_{M\in\mathrm{MSupp}(S/R)}(\pounds[R_M,S_M])$. Then:
\begin{enumerate}
 \item For each $M\in\mathrm{MSupp}(S/R),\ | \mathcal A_M|=1$. Let $T_{1,M}$ be the unique element of   $  \mathcal A_M$.
 
 \item  $\mathcal S[R,S]=\prod_{M\in\mathrm{MSupp}(S/R)} T_{1,M}$.
 
 \item For each $i\in \mathbb N_m$, the atoms of $[S_{i-1}, S]$ are the $T_{i,M}$ such that, for each $M,M'\in\mathrm{MSupp}(S/R),\ M'\neq M,\ (T_{i,M})_M=R_{M,i}$ if $i\leq n_M$ and $(T_{i,M})_M=S_M$ if $i> n_M$, with $(T_{i,M})_{M'}=R_{M',i-1}$. Moreover, $S_i=\prod_{M\in\mathrm{MSupp}(S/R)} T_{i,M}$. In particular, $\pounds[R,S]=m$.
 
 \item  $\ell[R,S]=\sum_{i=0}^{n-1}\ell[S_i,S_{i+1}]=\sum_{M\in\mathrm{MSupp}(S/R)}\pounds[R_M,S_M]$.
 \end{enumerate}
\end{proposition}

\begin{proof} 
Since $R\subset S$ is arithmetic, for each $M\in\mathrm{MSupp}(S/R),\ \{R_{M,i}\}_{i=0}^{n_M}$ is the Loewy series of $R_M\subset S_M$ by Corollary  \ref{8.31}.

(1) Let $M\in\mathrm{MSupp}(S/R)$. Since $R_M\subset S_M$ is a chain, $R_{M,1}$ is its only atom.  From Lemma \ref{8.321}, we deduce that  there is  $T_{1,M}\in[R,S]$ such that $R\subset T_{1,M}$ is minimal with $(T_{1,M})_M=R_{M,1}$ and $(T_{1,M})_{M'}=R_{M'}$ for  $M'\in\mathrm{MSupp}(S/R),\ M'\neq M \ (*)$. Since $R\subset S$ is integral, it is a $\mathcal B$-extension by \cite[Proposition 2.20]{Pic 10}, so that $T_{1,M}\in \mathcal A_M$  is the only element of $\mathcal A$ satisfying $(*)$. Then, $  \mathcal A_M=\{T_{1,M}\}$ and $| \mathcal A_M|=1$.

(2)  Proposition \ref{8.33} yields that $\mathcal S[R,S]=\prod_{M\in\mathrm{MSupp}(S/R)}(\prod_{A\in\mathcal A_M} A)= \prod_{M\in\mathrm{MSupp}(S/R)}T_{1,M}$.

(3) We proved in Proposition  \ref{8.8} that for $M\in\mathrm{Supp}(S/R)$ and $\{S'_i\}_{i=0}^{n_M}$ the Loewy series of $[R_M,S_M]$, then $S'_i=(S_i)_M$ for each $i\in\{0,\ldots,n_M\}$, where $n_M=\inf\{i\in\mathbb N_n\mid M\not\in\mathrm{Supp}(S/S_i)\}$. In particular, $S'_i=(S_i)_M=R_{M,i}$. Let $T$ be an  atom of $[S_{i-1}, S]$, and set $M:=\mathcal C(S_{i-1},T)\cap R\in\mathrm{MSupp}_R(S/S_{i-1})$, so that $i\leq n_M$ and  $S_{i-1}'\subset T_M$ is minimal. Moreover, by minimality of $S_{i-1}\subset T$,  we get $T_{M'}=(S_{i-1})_{M'}$ for $M'\in\mathrm{MSupp}(S/R),\ M'\neq M$. Since, $R\subset S$ is  a $\mathcal B$-extension, it follows that $T$ is unique for a given $M$, and of the form $T_{i,M}$ such that, for each $M,M'\in\mathrm{MSupp}(S/R),\ M'\neq M,\ (T_{i,M})_M=R_{M,i}$ if $i\leq n_M$ and $(T_{i,M})_M=S_M$ if $i> n_M$, and $(T_{i,M})_{M'}=R_{M'}$. Moreover, $S_i=\prod_{M\in\mathrm{MSupp}(S/R)} T_{i,M}$. In particular, $\pounds[R,S]=\sup_{M\in\mathrm{MSupp}(S/R)}(\pounds[R_M,S_M])=m$ by Corollary  \ref{8.82}.

(4) Since $R\subset S$ is arithmetic, it is distributive (Definition  \ref{4.2}), so that $\ell[R,S]=\sum_{i=0}^{n-1}\ell[S_i,S_{i+1}]$ by Proposition \ref{8.3}. But, $S_i\subset S_{i+1}$ is Boolean, so that $\ell[S_i,S_{i+1}]$ is the number of atoms of $S_i\subset S_{i+1}$ by \cite[Theorem 3.1]{Pic 10}. In view of (3), they are gotten, for each $S_i\subset S_{i+1}$, by the elements of the chain $[R_M,S_M]$ which are of the form $R_{M,i+1}$, that is for $i< n_M$. Then, $\ell[S_i,S_{i+1}]=|\{R_{M,i+1}\mid i< n_M\}|$, giving $\ell[R,S]=\sum_{M\in\mathrm{MSupp}(S/R)}(n_M-1)$.
\end{proof} 

 In Corollary \ref{8.31}, we proved that if $[R,S]$ is a chain, then $[R,S]=\{S_i\}_{i=0}^n$. 
We introduce the following property:

\begin{definition} \label{8.4} An FCP 
extension  $R\subset S$ with Loewy series $\{S_i\}_{i=0}^n$ is said to satisfy the  property $(\mathcal P)$ (or is a $\mathcal P$-extension) if $[R,S]=\cup_{i=0}^{n-1}[S_i, S_{i+1}]$. 
 (The Loewy series gives a partition of $[R,S]$.)
\end{definition}

Here is an example of such  a $\mathcal P$-extension.

 \begin{example} \label{8.41}  In \cite[Example 5.17 (2)]{Pic 6}, we  gave the following example: Set $k:=\mathbb{Q}, \ L:=k[x]$, where $x:=\sqrt 3+\sqrt 2$ and    $k_i:=k[\sqrt i],\ i=2,3,6$. Then, $[k,L]=\{k,k_1,k_2,k_3,L\}$ and  the following diagram holds:
$$\begin{matrix}
   {}  &        {}      & L             &       {}       & {}     \\
   {}  & \nearrow & \uparrow  & \nwarrow & {}     \\
k_1 &       {}       & k_2         &      {}        & k_3 \\
  {}   & \nwarrow & \uparrow & \nearrow & {}     \\
  {}   &      {}        & k             & {}             & {} 
\end{matrix}$$
so that $k\subset L$ is a non distributive extension of length 2 (because a diamond \cite[Theorem 4.7]{R}), with atoms $k_i,\ i=1,2,3$ so that $L=k_ik_j$, for any $i,j\in\{1,2,3\},\ i\neq j$. Then, $\mathcal S[k,L]=L=S_1,\ [k,L]=[k,S_1]$ and $k\subset L$ 
 is a $\mathcal P$-extension. 
\end{example} 
 
 Although many properties of $\mathcal P$-extensions  will be gotten for a distributive extension, we begin to give two results for  a non necessarily distributive $\mathcal P$-extension.

\begin{proposition} \label{8.42} Let $R\subset S$ be an FIP
 extension with Loewy series $\{S_i\}_{i=0}^n$. Then, $|[R,S]|\geq \sum_{i=0}^{n-1}|[S_i,S_{i+1}]|+1-n$, with equality if and only if $R\subset S$ 
  is a $\mathcal P$-extension. 
\end{proposition} 

\begin{proof} Since $ \cup_{i=0}^{n-1}[S_i,S_{i+1}]\subseteq[R,S]$, we get that $ |\cup_{i=0}^{n-1}[S_i,S_{i+1}]|\leq|[R,S]|$. But $\{S_i\}_{i=0}^n$ is an ascending chain, so that $[S_{i-1},S_i]\cap[S_i,S_{i+1}]=\{S_i\}$ for each $i\in\mathbb N_{n-1}$ and $[S_j,S_{j+1}]\cap[S_i,S_{i+1}]=\emptyset$ for any $(i,j)$ such that $j\neq i,i-1,i+1$. Then, $ |\cup_{i=0}^{n-1}[S_i,S_{i+1}]|=\sum_{i=0}^{n-1}|[S_i,S_{i+1}]|+1-n$, since there are $n-1$ elements $S_i$ common to two distinct subsets $[S_j,S_{j+1}]$. The equality holds if and only if  $ \cup_{i=0}^{n-1}[S_i,S_{i+1}]=[R,S]$  if and only if $R\subset S$ satisfies property ($\mathcal P$). 
\end{proof}

\begin{theorem} \label{8.5} Let $R\subset S$ be an FCP $\mathcal P$-extension with Loewy series $\{S_i\}_{i=0}^n$. Then $R\subset S$ is  distributive if and only if $S_i\subset S_{i+1}$ is Boolean for each  $0 \leq i\leq n-1$.
  If these conditions hold, then $R\subset S$ has FIP.
\end{theorem}

\begin{proof} One part of the proof is Proposition \ref{8.3}. 

Conversely, assume that $R\subset S$ satisfies property $(\mathcal P)$ and that $S_i\subset S_{i+1}$ Boolean for each $i\leq n$.  Let $T,U,V\in [R,S]$ be such that $UT=VT$ and $U\cap T=V\cap T$. We claim that  that $U=V$.  This will prove that $R\subset S$ is  distributive by  \cite[Theorem 1.6, page 9]{Cal}. The result is obvious if $S\in\{U,V,T\}$. Then, choose $i,j,k\in\{0,\ldots,n-1\}$ such that $T\in[S_i,S_{i+1}[,\ U\in[S_j,S_{j+1}[,\ V\in[S_k, S_{k+1}[$. Set $l:=\sup(j+1,k+1)$ and $l':=\inf(j,k)$. This yields that $l'<l$ and $S_{l'}\subseteq U,V\subset S_l$.

Consider the different cases:

(1) If $i=j=k$, then $U=V$, because $T,U,V\in[S_i, S_{i+1}]$, which is Boolean, and then distributive.

(2) Assume $i\geq l$, so that $U,V\subset S_l\subseteq  S_i\subseteq T$. Then, $U=U\cap T=V\cap T=V$.

(3) Assume that $i< l'$, so that $i+1\leq  l'$ which implies $T\subset S_{i+1}\subseteq S_{l'}\subseteq U,V$. Then, $U=UT=VT=V$.

(4) The last case  to consider (which has two subcases) is when $l'\leq i<l\ (*)$. There is no harm to assume $j\leq k$. In this case, $l'=j$ and $l=k+1$, so that $(*)$ yields $j\leq i\leq k$.  

If $j= i\leq k$, we can take $i<k$ because of (1), and then $i+1\leq k$. It follows that $U,T\subset S_{i+1}\subseteq S_k\subseteq V$. Then, $ V=VT=UT$ and $T=T\cap V=T\cap U$ leads to $T\subseteq U$, whence $V=UT=U$.

If $j<i\leq k$, then, by $j+1\leq i$, we obtain $U\subset S_{j+1}\subseteq S_i\subseteq T$ which gives $U\cap T=U=V\cap T\subset S_i\subseteq  S_k\subseteq V$, so that $S_i\subseteq T\cap V=U$, a contradiction which shows that this case does not occur. 

To conclude, $U=V$ in each case and $R\subset S$ is  distributive.

 The last result holds since an FCP distributive extension has FIP. 
\end{proof}

\begin{corollary} \label{8.50} Let $R\subset S$ be a distributive  FCP (hence FIP) $\mathcal P$-extension, with Loewy series $\{S_i\}_{i=0}^n$ and $T\in]R,S[$.There is some  $k\in\{0,\ldots,n-1\}$, such that $T\in[S_k,S_{k+1}[$ and then $\{S_i\}_{i=0}^k\cup\{T\}$ is the Loewy series of  $R\subset T$ and $\{T\}\cup\{S_i\}_{i=k+1}^n$  is the Loewy series of  $T\subset S$.  Moreover, $R\subset T$ and $T\subset S$  
 are  $\mathcal P$-extensions. 
\end{corollary}

\begin{proof} Since $R\subset S$ is a $\mathcal P$-extension,
 there is some $k\in\{0,\ldots,n-1\}$  such that $T\in[S_k,S_{k+1}[$, so that $S_k\subseteq T\subset S_{k+1}$. Let $\{T_i\}_{i=0}^m$ (resp. $\{T'_i\}_{i=0}^r$)  be the Loewy series of  $R\subset T$ (resp.   $T\subset S$). By definition of the socle of an extension, we have obviously $S_i=T_i$ for each $i\in\{0,\ldots,k\}$. Moreover, $S_k\subset S_{k+1}$ is  Boolean  by Proposition \ref{8.3}, which implies that $T$ is a product of atoms of $S_k\subset S$  \cite[Theorem 3.1]{Pic 10}. In fact, $T$ is a product of the atoms of $S_k\subset S$ contained in $T$, so that $T$ is the product of atoms of $S_k\subset T$, giving $T=\mathcal S[S_k,T]$. Then, $T_{k+1}=T$ and $m=k+1$. Now, $T=T'_0$. From \cite[Proposition 3.11]{Pic 10}, we deduce that  $T\subset S_{k+1}$ is also  Boolean, and $S_{k+1}$ is the product of the atoms of $[T, S_{k+1}]$. We claim that $ S_{k+1}=\mathcal S[T,S]$. Deny, so that there exists some atom $A$ of $[T,S]$ which is not in $ [T, S_{k+1}]$. But $T\subset A$ is a minimal extension. Since 
  $R\subset S$ is a $\mathcal P$-extension, we get that $A\in\cup_{i=k+1}^{n-1}[S_i,S_{i+1}]$ and then  $T\subset  S_{k+1}\subset A$, a contradiction. Then, $ S_{k+1}=T'_1$ and the other terms of the Loewy series of $T\subset S$ are the $S_i$ for $i\in\{k+2,\ldots,n\}$. The last property follows easily.
\end{proof} 

In the next result, we use the radical $\mathcal R$ of an extension, (Definition~\ref{8.1} (2)). 
\begin{corollary} \label{8.51} Let $R\subset S$ be a distributive  FCP 
$\mathcal P$-extension
 (hence FIP), with Loewy series $\{S_i\}_{i=0}^n$. Then   $S_i=\mathcal R[S_i, S_{i+1}]$  for each  $i\in\{0,\ldots,n-1\}$ and $S_{n-1}=\mathcal R[R, S]$. 
\end{corollary}

\begin{proof} In view of Theorem~\ref{8.5}, $S_i\subset S_{i+1}$ is Boolean for each  $i\in\{0,\ldots,n-1\}$. Then $S_i=\mathcal R[S_i, S_{i+1}]$  for each  $i\leq n-1$ by de Morgan's law \cite[Theorems 3.43 and 5.1]{R} and the equivalences of \cite [page 292]{St}.
   Indeed, since $S_i\subset S_{i+1}$ is Boolean, $S_i$ is the complement of $S_{i+1}$. Now, $S_{i+1}$ is the product of atoms of $[S_i, S_{i+1}]$, which implies that $S_i$ is the intersection of co-atoms of $[S_i, S_{i+1}]$; that is,  $\mathcal R[S_i, S_{i+1}]$. 

Since $R\subset S$ is a $\mathcal P$-extension, $[R,S]=\cup_{i=0}^{n-1}[S_i, S_{i+1}]$. Let $A\in\mathcal{CA}$. There is some $i\in\{0,\ldots,n-1\}$ with  $A\in[S_i, S_{i+1}]$, whence $A$ is comparable to any $S_i$. But $A\subset S$ is minimal, so that $S_{n-1}\subseteq A$ yields that $A$ is a co-atom of 
$[S_{n-1}, S]$.
 It follows that $S_{n-1}\subseteq\cap_{A\in\mathcal{CA}}A=\mathcal{R}[R,S]\subseteq\mathcal{R}[S_{n-1},S]=S_{n-1}$ 
 by 
 the first part of the proof.
\end{proof}

Proposition \ref{5.51} says that when $R\subseteq S$  has FCP, $T\in[R, S]$ is $\Pi$-irreducible if and only if either  $T=R$ or there is a unique $T' \in[R,S]$ such that $T'\subset T$ is  minimal.  For an FCP distributive  $\mathcal P$-extension, these elements can be characterized thanks to the Loewy series. In fact, the following theorem characterizes FCP distributive 
$\mathcal P$-extensions.

\begin{theorem} \label{8.6} Let $R\subset S$ be a distributive FCP 
 (hence FIP) extension with Loewy series $\{S_i\}_{i=0}^n$. Then, $R\subset S$ is a $\mathcal P$-extension  if and only if the following condition holds:  any $T\in[R,S]$ is $\Pi$-irreducible in $[R,S]$ if and only if there exists some $i\leq n-1$ such that $T$ is an atom of $[S_i, S_{i+1}]$. 
\end{theorem}

\begin{proof} Assume first 
$R\subset S$ is a $\mathcal P$-extension. Let $T\in[R,S]$ be $\Pi$-irreducible in $[R,S]$. There is some $i\in\{0,\ldots,n-1\}$ such that $T\in [S_i, S_{i+1}]$, which is Boolean. It follows that $T$ is a product of 
 $m$
  atoms of $[S_i, S_{i+1}]$
   by \cite[Theorem 3.1]{Pic 10}.
 But, $T$ being $\Pi$-irreducible 
  in $[R,S]$ is a fortiori  irreducible in $[S_i, S_{i+1}]$, giving $m=1$, so that 
  $T$ is itself an atom of $[S_i, S_{i+1}]$. 
   Conversely, let $T$ be an atom of some  $[S_i, S_{i+1}]$. We show that $T$ is $\Pi$-irreducible in $[R,S]$. Deny, and let $U,V\in[R,S]\setminus\{T\}$ be such that $T=UV$, so that $U,V\subset T\subseteq S_{i+1}\ (*)$. Because of property $(\mathcal P)$, there exist $j,k\in\{0,\ldots,n-1\}$ such that $U\in[S_j, S_{j+1}[$ and $V\in[S_k, S_{k+1}[$. Moreover, $(*)$ implies that $j,k\leq i$. We cannot have $j=k=i$ since $T$ is an atom of $[S_i, S_{i+1}]$. Assume that only one $U,V$ is in $[S_i, S_{i+1}]$, for instance, $U\in[S_i, S_{i+1}]$, so that $i=j$. Then, $V\subset S_i\subseteq U\subset T$ leads to $UV=U\subset T$, a contradiction. For the remaining case $j,k<i$, we have $j+1,k+1\leq i$ and $U,V\not\in[S_i, S_{i+1}]$. We get that $UV\subseteq S_i\subset T$, again a contradiction. Then, $T$ is $\Pi$-irreducible.

Now, we show that $[R,S]=\cup_{i=0}^{n-1}[S_i, S_{i+1}]$ if the following condition holds: any $T\in[R,S]$ is $\Pi$-irreducible if and only if there exists some $i\in\{0,\ldots,n-1\}$ such that $T$ is an atom of $[S_i, S_{i+1}]$. Let $U\in [R,S]$, so that $U=\prod_{j=1}^mT_j$, 
 for some positive integer $m$, where $T_j$ is $\Pi$-irreducible for each $j$ by \cite[Proposition 2.9]{Pic 10}. The hypothesis gives that for each $j$, there exists 
 a unique 
 $i_j\in\{0,\ldots,n-1\}$ such that $T_j$ is an atom of $[S_{i_j}, S_{i_j+1}]$. 
 Set $k:=\sup\{i_j\mid j\in\mathbb N_m\},\ I_1:=\{j\in\mathbb N_m\mid i_j<k\}$ and $ I_2:=\{j\in\mathbb N_m\mid i_j=k\}$. In particular, $I_2\neq \emptyset$.
 Then,  for each $h\in I_1$ and for each $l\in I_2$, we have 
 $T_h\subseteq S_k\subset T_l\subseteq S_{k+1}$, so that $U=(\prod_{h\in I_1}T_h)(\prod_{l\in I_2}T_l)=\prod_{l\in I_2}T_l$. Then $U\in[S_k, S_{k+1}]$ and property $(\mathcal P)$ holds.
\end{proof}

 The following definitions are needed for our study.
  
  \begin{definition}\label{1.3} An integral extension $R\subseteq S$ is called {\it infra-integral} \cite{Pic 2} (resp.;  {\it subintegral} \cite{S}) if all its residual extensions $\kappa_R(P)\to \kappa_S(Q)$, (with $Q\in\mathrm {Spec}(S)$ and $P:=Q\cap R$) are isomorphisms (resp$.$; and the natural map $\mathrm {Spec}(S)\to\mathrm{Spec}(R)$ is bijective). An extension $R\subseteq S$ is called {\it t-closed} (cf. \cite{Pic 2}) if the relations $b\in S,\ r\in R,\ b^2-rb\in R,\ b^3-rb^2\in R$ imply $b\in R$. The $t$-{\it closure} ${}_S^tR$ of $R$ in $S$ is the smallest element $B\in [R,S]$  such that $B\subseteq S$ is t-closed and the greatest element  $B'\in [R,S]$ such that $R\subseteq B'$ is infra-integral. An extension $R\subseteq S$ is called {\it seminormal} (cf. \cite{S}) if the relations $b\in S,\ b^2\in R,\ b^3\in R$ imply $b\in R$. The {\it seminormalization} ${}_S^+R$ of $R$ in $S$ is the smallest element $B\in [R,S]$ such that $B\subseteq S$ is seminormal and the greatest  element $ B'\in[R,S]$ such that $R\subseteq B'$ is subintegral. 
  The {\it canonical decomposition} of an arbitrary ring extension $R\subset S$ is $R \subseteq {}_S^+R\subseteq {}_S^tR \subseteq \overline R \subseteq S$. 
 \end{definition}  
 
   Next proposition describes the link between the elements of the canonical decomposition and minimal extensions.

\begin{proposition}\label{1.31} \cite[Proposition 4.5]{Pic 6} Let there be an integral extension $R\subset S$ and a maximal chain $\mathcal C$ of $ R$-subextensions of $S$, defined by $R=R_0\subset\cdots\subset R_i\subset\cdots\subset R _n= S$, where each $R_i\subset R_{i+1}$ is  minimal.  The following statements hold: 

\begin{enumerate}
\item $R\subset S$ is subintegral if and only if each $R_i\subset R_{i+1}$ is  ramified. 

\item $R\subset S$ is seminormal and infra-integral if and only if each  $R_i\subset R_{i+1}$ is decomposed. 

\item $R \subset S$ is t-closed if and only if  each $R_i\subset R_{i+1}$ is inert. 
\end{enumerate}
If either (1) or (3) holds, then $\mathrm {Spec}(S)\to\mathrm{Spec}(R)$ is bijective.
\end{proposition}

\begin{example} \label{8.7} We now give three examples of an FCP distributive, 
 hence FIP
 (and not Boolean) extension where  property $(\mathcal P)$ holds or not. 

(1) Set $G:=\mathbb Z/12\mathbb Z$, which is a cyclic group, and let $k\subset L$ be a cyclic extension with Galois group $G$. The proper subgroups of $G$ are $2G,3G,4G$ and $6G$, so that the lattice $\mathcal G$ of subgroups of $G$ is $\{0,2G,3G,4G,6G,G\}$, which is distributive \cite[Exercise 15, page 125]{R}. Using the order isomorphism of lattices $\psi:{\mathcal G}\to [k,L]$ defined by $\psi(H):=$Fix$(H)$, we obtain the following lattice $[k,L]=\{k,$Fix$(iG),L\mid i=2,3,4,6\}$. Set $L_i:=\psi(iG)$. We have the following diagram: 

$$\begin{matrix}
   {}  &        {}      & L_2 &       \to      & L_4  & {}            & {}  \\
   {}  & \nearrow & {}     & \searrow & {}       & \searrow & {} \\
k     &       {}       & {}     &      {}       & L_6    & \to          & L   \\
  {}   & \searrow & {}    & \nearrow & {}        &  {}          & {}  \\
  {}   &      {}       & L_3 & {}            & {}         & {}          & {}
\end{matrix}$$

Then, $k\subset L$ is distributive but does not satisfy property $(\mathcal P)$. Indeed, $L_6$ is the socle of $k\subset L$, because $L_2$ and $L_3$ are the atoms of $k\subset L$, with $L_6\subset L$ minimal, so that  $S_0=k,\ S_1=L_6$ and $S_2=L$. Moreover, $L_4\not\in[S_0,S_1]\cup[S_1,S_2]$. In particular, we cannot apply Corollary \ref{8.51}. Indeed,  $\mathcal R[k,L]=L_4\cap L_6\neq S_1=L_6$.

 This example shows that the results of Theorem \ref{8.5} hold even if property $(\mathcal P)$ is not satisfied, since $S_0\subset S_1$ and $S_1\subset S_2$ are Boolean. This also shows that in a distributive extension,  a $\Pi$-irreducible element is not necessarily an atom of some $[S_j,S_{j+1}]$ (see $L_4$), and an atom of some $[S_j,S_{j+1}]$ is not necessarily $\Pi$-irreducible  (see $L$).

(2) An obvious example of a distributive 
 $\mathcal P$-extension
 $R\subset S$ is when $[R,S]$ is a chain.

(3) We give here a  more involved example of a distributive  $\mathcal P$-extension. Take $k\subset L$  a finite separable field extension of degree 6, such that $[k,L]$ is a Boolean algebra with $[k,L]=\{k,L_1,L_2,L\}$ and $\ell[k,L]=2$ \cite[Example 4.17]{Pic 10}. Set $S:=L[X]/(X^2)$, so that $L\subset S$ is a minimal ramified extension and $S$ is a local ring, with maximal ideal $M$. Set $T:={}_S^tk$ which is a local ring 
 because $k\subset S$ is a finite integral extension. 
Then, $T\subset S$ is t-closed with  $M=(T:S)\in\mathrm{Max}(T)$ (\cite[Lemma 3.17]{DPP3}). Since $k\subset S$ has FCP by \cite[Theorem 4.2]{DPP2}, there exists a finite chain $k:=R_0\subset\cdots\subset R_n\subset R_{n+1}:=T$ such that $R_i\subset R_{i+1}$ is minimal ramified for each $i=0,\ldots,n$. There are no decomposed minimal extension $R_i\subset R_{i+1}$ because $T$ is local. Set $R:=R_n$ and consider the extension $R\subset S$. Since $T/M\cong k$ and $S/M\cong L$, we get that $T\subset S$ is a Boolean extension by \cite[Proposition 3.4 (5)]{Pic 10} with $[T,S]=\{T,T_1,T_2,S\}$, where $T_i$ is such that $T_i/M=L_i$  because of the bijection $[T,S]\to[k,L]$ given by $U\mapsto U/M$. Moreover, $\ell[R,S]=3$ by \cite[Proposition  3.2]{Pic 4}. We claim that $T$ is the only atom of the extension. Assume there exists some $T'\in[R,S]\setminus\{T\}$ such that $R\subset T'$ is minimal. We get that $R\subset T'$ can be neither decomposed (as we already observed since $S$ is local) nor ramified, because in this case, we should have $T'\subset T$, a contradiction. If $R\subset T'$ is minimal inert, this leads also to a contradiction, because $T$ would not be the t-closure, since some minimal ramified extensions would start from $T$ (\cite[Proposition 7.4]{DPPS}). Then, $T$ is the socle $S_1$ of the extension. Moreover $S=S_2$ since $T\subset S$ is Boolean. Consider some $U\in]R,S]$ and let $V\in[R,U]$ be such that $R\subset V$ is minimal. As we already observed, $V=T$ so that $U\in[T,S]=[S_1,S_2]$, which yields $[R,S]=[S_0,S_1]\cup[S_1,S_2]$. 
We have the following diagram:

$$\begin{matrix}
  {}  & {}   & {}          &       {}       & T_1 & {}            & {}   \\
  {}  & {}   & {}          & \nearrow & {}     & \searrow & {}   \\
 R   & \to & T=S_1  &      {}       & {}     & {}       & S=S_2 \\
 {}   & {}   & {}          & \searrow & {}     & \nearrow  & {}  \\
 {}   & {}   & {}          &      {}       & T_2  & {}             & {}
\end{matrix}$$ 
Then, $R\subset S$ is a $\mathcal P$-extension and $R\subset S$ is distributive by Corollary \ref{8.5}. \end{example}

Let $R\subset S$ be a ring extension. We recall that    
 $R$ is called {\it unbranched } in $S$ if $\overline R$ is local. In this case, if $R \subset S$ has FCP, the following Lemma shows that any $T\in [R,S]$ is local. 
  An extension $R\subset S$ is called {\it quasi-Pr\"ufer} if $\overline R\subseteq S$ is a Pr\"ufer extension; that is, $\overline R\subseteq T$ is a flat epimorphism for each $T\in [\overline R,S]$ \cite[Definition 2.1]{Pic 5}. An FCP extension is  quasi-Pr\"ufer \cite[Corollary 3.3]{Pic 5} since an FCP integrally closed extension is Pr\"ufer \cite[Proposition 1.3]{Pic 5}.

 \begin{lemma} \label{8.71} Let $R \subset S$ be 
  a   quasi-Pr\"ufer extension
   such that $R$ is unbranched  in $S$. Then, $T$ is local for each $T\in [R,S]$. 
   In particular, this holds if $R \subset S$ has FCP.
\end{lemma} 

\begin{proof} Since $\overline R$ is local, so is any element of $[R,\overline R]$. Let $T\in [R,S]$, so that $\overline R^T=\overline R\cap T$ is local. Since $R \subset S$ 
 is quasi-Pr\"ufer, so is $R \subset T$, and $T$ is local by \cite[Proposition 1.2 (3)]{Pic 5}, 
because $\overline R^T\subseteq T$ is Pr\"ufer. 
\end{proof} 

 As we saw in Example \ref{8.7} (3), the t-closure of the extension is the socle of the extension. We are going to show that for some distributive extensions $R\subset S$, the t-closure and the integral closure are elements of the Loewy series.

\begin{proposition} \label{8.9} Let $R\subset S$ be an FCP distributive extension 
 (hence FIP) 
 such that $R$ is unbranched  in $S$,
 and  $\{S_i\}_{i=0}^n$ its Loewy series. Then,  there exist $k,l\in\{0,\ldots,n\}$ such that ${}_{\overline R}^tR=S_k$ and $\overline R=S_l$. Moreover, $[R,{}_{\overline R}^tR]$ and $[\overline R,S]$ are chains.
\end{proposition}  
\begin{proof} Since $\overline R$ is a local ring, so are $R$ and  $S_i$ for any $i\in\mathbb N_{n+1}$ 
 by Lemma \ref{8.71}. 
 Let $M$ be the maximal ideal of $R$. 

 Since $S_i$ is local and $S_i\subset S_{i+1}$ is Boolean  FIP,
  it follows from \cite[Corollary 3.19]{Pic 10}  that $S_i\subset S_{i+1}$ is either integral or Pr\"ufer. Assume that $S_i\subset S_{i+1}$ is Pr\"ufer for some 
 $i\in\{0,\ldots,n-2\}$
  and let $j>i$ for $j<n$. We claim that $S_j\subset S_{j+1}$ is  Pr\"ufer. Deny, so that $S_i\subset S_{j+1}$ is neither integral, nor Pr\"ufer. Set $S'_i:=\overline {S_i}^{S_{j+1}}$. Then, $S'_i\neq S_i,S_{j+1}$. It follows that there exist a minimal integral extension $S_i\subset T$ and a minimal Pr\"ufer extension $S_i\subset U$ with $T,U\in[S_i,S_{j+1}],\ T\subseteq S'_i,\ U\subseteq S_{i+1}$. An application  of \cite[Lemma 1.5]{Pic 6}Ê leads to 
  $\mathcal C(S_i,T)\neq \mathcal C(S_i,U)$,
  a contradiction since $S_i$ is local. Then, $S_j\subset S_{j+1}$ is  Pr\"ufer for any $j>i$. In particular,  as soon as some $S_k\subset S_{k+1}$ Pr\"ufer, so is $S_i\subset S_{i+1}$ for each $i\geq k$, whence  $S_k=\overline R$. 

Now, if $R\neq\overline R$, we can work with the extension $R\subset\overline R$, which is also distributive, and its Loewy series is $\{S_i\}_{i=0}^{k}$. From \cite[Proposition 3.24]{Pic 10},  we deduce that $S_i\subset S_{i+1}$ is either infra-integral (more precisely subintegral since $S_{i+1}$ is local), or t-closed, for each $i<k$ since $S_i$ is local.
In order to establish  the result for  the t-closure, we mimic the previous proof given for the integral closure.  Assume that $S_i\subset S_{i+1}$ is t-closed and let $j>i$ for $j<k$. We claim that $S_j\subset S_{j+1}$ is  t-closed. Deny, so that $S_i\subset S_{j+1}$ is neither infra-integral, nor t-closed. Set $S'_i:={}_{S_{j+1}}^tS_i$. Then, $S'_i\neq S_i,S_{j+1}$. It follows that there exist a minimal ramified extension $S_i\subset T$ and a minimal inert extension $S_i\subset U$ with $T,U\in[S_i,S_{j+1}],\ T\subseteq S'_i,\ U\subseteq S_{i+1}$. Then \cite[Proposition 7.4 ]{DPPS} shows that there are two maximal chains in $[S_i,TU]$ of different lengths, contradicting the distributivity of $R\subset S$ by Proposition \ref{1.0}. Then, $S_j\subset S_{j+1}$ is  t-closed for any $j$ such that $k>j>i$. In particular,  as soon as  $S_l\subset S_{l+1}$ is t-closed, so is $S_i\subset S_{i+1}$ for each $i\geq l$, giving $S_l={}_{\overline R}^tR$. 

Since $\overline R$ is a local ring, and  $R\subseteq S_l$ is subintegral, so is $S_i\subset S_{i+1}$ for each $i<l$, and then  is minimal ramified by \cite[Lemma 3.26]{Pic 10}. In particular, each $S_i\subset S_{i+1}$ has only one atom, which is $S_{i+1}$. Let $T\in]R,S_l[$. Since $R\subset T$ has FCP, there is a tower $R\subset T_i\subset T$, where $\{T_i\}_{i=0}^r$ is the Loewy series of $R\subset T$. Since $R=S_0=T_0$, an obvious induction shows that $T_i=S_i$ for all $i<r$, and  $T=S_r$   for some $r\leq l$, so that $R\subseteq {}_{\overline R}^tR$ is a chain. By \cite[Theorem 6.10]{DPP2}, $\overline R\subset S$ is a chain. 
\end{proof}

 \begin{remark} \label{8.91} Example \ref{8.81} shows that the conclusion of Proposition \ref{8.9} does not hold in general. In this example, $R$ is not local (although $R\subset S$ is spectrally injective) and ${}_S^tR=R_1\neq S_i$ for $i=1,2$ because $R\subset R_1$ is minimal ramified and $R_1\subset S$ is t-closed. A more precise study will be made in a forthcoming paper.
\end{remark} 

When an FCP distributive extension  satisfies property $(\mathcal P)$, the Loewy series allows to give information about the extension.

\begin{proposition} \label{8.10} Let $R\subset S$ be an FCP distributive  (hence FIP)
$\mathcal P$-extension and  $\{S_i\}_{i=0}^n$ its Loewy series, so that each  $S_i\subset S_{i+1}$ is Boolean.
Then

\begin{enumerate}
\item There exist $j,k\in\{0,\ldots,n\}$ with $j\leq k$ such that ${}_{\overline R}^tR\in[S_j,S_{j+1}[$ and $\overline R\in[S_k,S_{k+1}[$.

\item Let $i\in\{0,\ldots,n-1\}$. If $i<j,\ S_i\subset S_{i+1}$ is  infra-integral. If $j<i<k,\ S_i\subset S_{i+1}$ is  t-closed. If $i> k,\ S_i\subset S_{i+1}$ is  Pr\"ufer.

\item For $i\in\{0,\ldots,n-1\}$, let ${\mathcal A}_i$ be the set of atoms of $[S_i,S_{i+1}]$. Then $\ell[R,S]=\sum_{i=0}^{n-1}|{\mathcal A}_i|$.
\end{enumerate}
\end{proposition} 

\begin{proof}  Each $S_i\subset S_{i+1}$ is Boolean because of Proposition \ref{8.3}.

(1) Since $[R,S]=\cup_{i=0}^{n-1}[S_i, S_{i+1}]$, there exist $j,k\in\{0,\ldots,n\}$ such that ${}_{\overline R}^tR\in[S_j,S_{j+1}[$ and $\overline R\in[S_k,S_{k+1}[$. Moreover, $j\leq k$ since $S_j\subseteq {}_{\overline R}^tR\subseteq\overline R\subset S_{k+1}$ implies $j<k+1$. 

(2) Let $i<k$ so that $i+1\leq k$ which implies $S_i\subset S_{i+1}\subseteq S_k\subset\overline R$. Then $S_i\subset S_{i+1}$ is  integral. If $i<j$, then $S_i\subset S_{i+1}\subseteq S_j\subseteq {}_{\overline R}^tR$ and $S_i\subset S_{i+1}$ is  infra-integral. If $j<i<k$, then $j+1\leq i$, so that $ {}_{\overline R}^tR\subset S_{j+1}\subseteq  S_i\subset S_{i+1}\subseteq S_k\subseteq \overline R$ and $S_i\subset S_{i+1}$ is  t-closed.

Let $i> k$, so that $i\geq k+1$ which implies $\overline R\subset S_{k+1}\subseteq S_i\subset S_{i+1}$. Then $S_i\subset S_{i+1}$ is  Pr\"ufer.

(3) By \cite[Lemma 4, p.486]{BM}, we get that $\ell[R,S]$ is the number of $\Pi$-irreducible elements of $[R,S]$. Then, Theorem \ref{8.6} gives the result.
\end{proof}

 We end this section by studying the Loewy series of some special extensions.
We need the following lemma used in the next Proposition.  

\begin{lemma}\label{8.11} \cite[Lemma 2.9 and the paragraph before Proposition 2.3]{Pic 3} An FCP extension $R \subset S$ with a factorization  $R \subseteq T \subseteq  S$ such that $\mathrm{Supp}_R(T/R) \cap \mathrm{Supp}_R(S/T)= \emptyset$  admits a unique factorization $R \subseteq U \subseteq S$ such that $T\cap U=R$ and $TU=S$ 
($U$ is the complement of $T$ in $[R,S]$). 
Moreover, $\mathrm{Supp}_R(U/R)=\mathrm{Supp}_R(S/T)$ and $\mathrm{Supp}_R(S/U)=\mathrm{Supp}_R(T/R)$. 
\end{lemma} 

 \begin{proposition}\label{8.12} Let $R\subset S$ be an FCP extension with Loewy series $\{S_i\}_{i=0}^n$ and $T\in]R,S[$  such that $\mathrm{Supp}_R(T/R) \cap \mathrm{Supp}_R(S/T)= \emptyset$. Then:
 \begin{enumerate}
 \item   $T$ has a unique complement $U\in[R,S]$.
 
  \item Let $\{T_i\}_{i=0}^m$ (resp.;  $\{U_i\}_{i=0}^r$) be the Loewy series of $[R,T]$ (resp.;  of $[R,U]$). Then $S_i=T_iU_i$ for each $i\in\{0,\ldots,n\}$, with $n=\pounds[R,S]=\sup(\pounds[R,T],\pounds[R,U])=\sup(m,r)$.
  \end{enumerate}
\end{proposition}

\begin{proof}
(1)  Lemma  \ref{8.11} gives $T\cap U=R$ and $TU=S\ (*)$. Moreover $U$ is unique satisfying these properties. 

(2) Set $\mathcal M:=\mathrm{MSupp}_R(S/R),\ \mathcal M_T:=\mathrm{MSupp}_R(T/R)$ and $\mathcal M_U:=\mathrm{MSupp}_R(U/R)=\mathrm{MSupp}_R(S/T)$. Then, $\mathcal M=\mathcal M_T\cup\mathcal M_U$ with $\mathcal M_T\cap \mathcal M_U=\emptyset\ (**)$ by Lemma \ref{8.11}. Now, we use Proposition  \ref{8.8}.

Let $M\in\mathcal M$ and $\{S'_i\}_{i=0}^{n_M}$ (resp.; $\{T'_i\}_{i=0}^{m_M}$,Ê$\{U'_i\}_{i=0}^{r_M}$) be the Loewy series of $[R_M,S_M]$ (resp.; $[R_M,T_M],\ [R_M,U_M]$). Then $S'_i=(S_i)_M$ for each $i\in\{0,\ldots,n_M\},\ T'_i=(T_i)_M$ for each $i\in\{0,\ldots,m_M\}$ and $U'_i=(U_i)_M$ for each $i\in\{0,\ldots,r_M\}$. In view of $(**)$, we have either $M\in\mathcal M_T$ (a), or $M\in\mathcal M_U$ (b). In case (a), $M\not\in \mathcal M_U$, so that $U_M=R_M$ and $T_M=S_M$. It follows that $n_M=m_M$ and $S'_i=T'_i$ for each $i\leq n_M$. Moreover, $U'_i=R_M$ for each $i\leq n_M$, so that $S'_i=T'_iU'_i=(S_i)_M=(T_i)_M(U_i)_M=(T_iU_i)_M$. The same reasoning shows that in case (b), $n_M=r_M$ and $(S_i)_M=(T_i)_M(U_i)_M=(T_iU_i)_M$, so that  $(S_i)_M=(T_iU_i)_M$ for any $M\in\mathrm{MSupp}_R(S/R)$.
From  Corollary \ref{8.82}, we deduce that  $n=\sup_{M\in\mathcal M}(n_M)=\sup[\sup_{M\in\mathcal M_T}(n_M),\sup_{M\in\mathcal M_U}(n_M)]=\sup[\sup_{M\in\mathcal M_T}(m_M),\sup_{M\in\mathcal M_U}(r_M)]=\sup(m,r)$. 

If $m=r$, then $n=m=r$. Let $i\leq n$. Then, $(S_i)_M=(T_i)_M(U_i)_M$ for each $M\in\mathcal M$ leads to $S_i=T_iU_i$.

If $m\neq r$, assume $m<r$, so that $n=r$. 
 As above, $S_i=T_iU_i\in[T,S]$ for each $i\leq m$.
  Recall that $T_i=T$ for each $i\in\{m,\ldots,r\}$. In particular, $S_m=T_mU_m=TU_m\in[T,S]$ and we still have $S_i=T_iU_i\in[T,S]$ for each $i\geq m$. 
   Moreover,
    $\mathrm{MSupp}_R(S/S_i)\subseteq \mathrm{MSupp}_R(S/T)=\mathcal M_U$ for each $i\geq m$. 
\end{proof} 

 In \cite[Definition 4.1]{Pic 5}, we call an extension $R\subset S$  {\it  almost-Pr\"ufer} if it can be factored $R\subseteq U\subseteq S$, where $R\subseteq U$ is Pr\"ufer and $U\subseteq S$ is integral. Actually, $U$ is the {\it Pr\"ufer hull} $\widetilde{R}$ of the extension.

\begin{corollary}\label{8.13} Let $R\subset S$ be an FCP almost-Pr\"ufer  extension with Loewy series $\{S_i\}_{i=0}^n$.  Let $\{T_i\}_{i=0}^m$ (resp.  $\{U_i\}_{i=0}^r$) be the Loewy series of $[R,\overline R]$ (resp.;  of $[R,\widetilde{R}]$). Then $S_i=T_iU_i$ for each $i\in\{0,\ldots,n\}$, with $n=\pounds[R,S]=\sup(\pounds[R,\overline R],\pounds[R,\widetilde{R}])=\sup(m,r)$.
\end{corollary}

\begin{proof} We have $\mathrm{Supp}_R(\overline R/R) \cap \mathrm{Supp}_R(S/\overline R)= \emptyset$  \cite[Proposition 4.16]{Pic 5}, with $\widetilde{R}$ the unique $U\in[R,S]$ such that $\overline R\cap U=R$ and $\overline U=S$.
 Then, Proposition \ref{8.12} gives the result.
\end{proof} 

 Given a ring $R$, recall that its  {\it Nagata ring} $R(X)$ is the localization $R(X) = T^{-1}R[X]$ of the ring of polynomials $R[X]$ with respect to the multiplicatively closed subset $T$ of all  polynomials with  content $R$.
In \cite[Theorem 32]{DPP4}, Dobbs and the authors proved that when $R\subset S$ is an extension, whose Nagata extension $R(X)\subset S(X)$  has FIP, the map $\varphi:[R,S]\to [R(X), S(X)]$ defined by $\varphi(T)= T(X)$ is an order-isomorphism. We show now that this map send the Loewy series of $R\subset S$ to the Loewy series of $R(X)\subset S(X)$.

\begin{proposition}\label{8.14} Let $R\subset S$ be an   extension such that $R(X)\subset S(X)$  has FIP. If   $\{S_i\}_{i=0}^n$ is the Loewy series of $[R,S]$, then $\{S_i(X)\}_{i=0}^n$ is the Loewy series of $[R(X),S(X)]$. In particular, $\pounds[R,S]=$
 $\pounds[R(X),S(X)]$.
\end{proposition}

\begin{proof} Since the map $\varphi:[R,S]\to [R(X), S(X)]$ defined by $\varphi(T)= T(X)$ is an order-isomorphism, it is also a lattice isomorphism. In particular, if $T,U\in[R,S]$ is such that $T\subset U$ is minimal, so is $T(X)\subset U(X)$ \cite[Theorem 3.4]{DPP3}. Let $\{S'_i\}_{i=0}^m$ be the Loewy series of $[R(X),S(X)]$. Then, an obvious induction on $i$ shows that $S'_i=S_i(X)$ for each $i\in\{0,\ldots,n\}$ since an atom $A$ of $S_i\subset S$ gives the atom $A(X)$ of $S_i(X)\subset S(X)$, and all atoms of $S_i(X)\subset S(X)$ are of this form. In particular, $n=\pounds[R,S]=m=\pounds[R(X),S(X)]$.
\end{proof} 

 \begin{remark}\label{8.141} If $R\subset S$ is an FCP extension such that $R(X)\subset S(X)$ is distributive, then $R(X)\subset S(X)$ has FIP because it has FCP by \cite[Theorem 3.9]{DPP3}. In this case, $R\subset S$ is distributive \cite[Proposition 3.6]{Pic 10}, and then has FIP.
\end{remark} 

 In Corollary \ref{8.112} and in Remark \ref{8.113}, we proved that for a ring extension $R\subset S$, the two socles $\mathcal S[R,S]$ and $\mathcal {MS}[R,S]$ may differ. Using idealization,  we may associate to some modules a ring extension. The Loewy length of some modules can be computed as the Loewy length of a ring extension. 
  Let $M$ be an $R$-module. The Loewy length of the $R$-module $M$ is denoted by $\lambda(M)$. We  consider the ring extension $R\subseteq R(+)M$, where $R(+)M$ is the idealization of $M$ in $R$.  

Recall that $R(+)M:=\{(r,m)\mid (r,m)\in R\times M\}$ is a commutative ring whose operations are defined as follows: 

$(r,m)+(s,n)=(r+s,m+n)$ \ \   and  \ \ \ $(r,m)(s,n)=(rs,rn+sm)$

Then  $(1,0)$ is the unit of $R(+)M$, and $R\subseteq R(+)M$ is a ring morphism defining $R(+)M$ as an $R$-module, so that we can identify any $r\in R$ with $(r,0)$. 

\begin{proposition}\label{8.15} Let $M$ be an $R$-module with finite length and let $S:=R(+)M$ be the idealization of $M$. Let $\{S_i\}_{i=0}^n$ be  the Loewy series of the ring extension $[R,S]$. Then $S_i=R(+)M_i$ for each $i\in\{0,\ldots,n\}$ where $\{M_i\}_{i=0}^n$ is the Loewy series of the lattice $\Lambda(M)$.  In particular, $\pounds[R,S]=\lambda(M)$.
\end{proposition}

\begin{proof} Since $M$ is an $R$-module with finite length, $R\subseteq R(+)M$ is an FCP quadratic extension   \cite[Propositions 2.2 and 2.3]{Pic 8} and \cite[Lemma 2]{HP}. Moreover, there is an order isomorphism $\psi:\Lambda(M)\cong [R,R(+)M]$ given by $\psi(N)=R(+)N$. In particular, for $N,N'\in\Lambda(M)$ such that $N\subset N'$, \cite[Proposition 2.8]{Pic 8} says that $R(+)N\subset R(+)N'$ is minimal if and only if $N'/N$ is a simple $R$-module. At last, for $N,N'\in\Lambda(M)$, obviously $(R(+)N) (R(+)N')=R(+)(N+N')$, so that $\psi$ is also a lattice isomorphism. Let $\{S_i\}_0^m$ be the Loewy series of $R\subseteq R(+)M$. Then, an easy induction on $i$ shows that $S_i=R(+)M_i$ for each $i\in\{0,\ldots,n\}$, where $\{M_i\}_{i=0}^n$ is the Loewy series of the lattice $\Lambda(M)$. Indeed, an atom $N'$ of $M/M_i$ is of the form $N/M_i$, where $N\in \Lambda(M)$ and  $M_i\subset N$, and gives the atom $R(+)N$ of $R(+)M_i\subset R(+)M$. Moreover, all atoms of $R(+)M_i\subset R(+)M$ are of this form.
 In particular, $\pounds[R,S]=\lambda(M)$.
\end{proof} 

\section{Finite  (distributive) field extension}

We first consider a finite field extension $k\subset L$ with separable closure $T$ and radicial closure $U$. The Loewy series of $k\subset L$ is linked to those of $k\subset T$ and  $k\subset U$. 
 If $k\subset L$ is a radicial extension, then $\mathrm{c}(k)$ is a prime number. 
 We recall that a  minimal field extension is either radicial, or separable \cite[Remark before Proposition 2.2, page 371]{Pic}. If $K$ is an atom of $[k,L]$, then $k\subset K$ is either radicial or separable. In the first case, we say that $K$ is a {\it radicial atom}, and in this case, $[K:k]=p=\mathrm{c}(k)$. In the second case, $K$ is a {\it separable atom}. 
 If $k\subset L$ is a finite field extension which is   not separable, let $T$ be its separable closure, so that $T\subset L$ is a radicial extension. In particular, $p:=\mathrm{c}(T)$ is a prime number, so that $\mathrm{c}(k)=p$.
Since a finite dimensional separable field extension has FIP, we consider in this section 
mainly FIP field extension. We found less results  for FCP not FIP field extensions.
 We begin with the following lemma. 

 \begin{lemma} \label{9.1} Let $k\subset L$ be an 
 FCP
  field extension with separable closure $T$,  radicial closure $U$ such that $T,U\not\in\{ k,L\}$. Let $T'\in[k,T]$ and $U'\in[k,U]$. Set $K:=T'U'$. Then: 
\begin{enumerate}
 \item $U'\subset K$ is separable and $T'\subset K$ is  radicial.

 \item If there exists $U''\in[k,U]$  such that  $ U''$ is a 
   radicial atom of $[U',U]$, then $KU''$ is 
     a radicial atom of $[K,L]$. Moreover, if $k\subset L$ has FIP, then  $KU''$ is the only radicial atom of $[K,L]$.

 \item The 
  separable atoms $V$ of   $[K,L]$ 
  are of the form $V=KT''$, where $T''$ is an atom of $T'\subset T$. In particular, $T''=V\cap T$.
 \end{enumerate}
 \end{lemma} 

\begin{proof}  Since $k\subset U$ is radicial, $p:=\mathrm{c}(k)$ is a prime number.

(1) Obvious, because $T'$ (resp.: $U'$) is generated by a separable  element 
(resp.; radicial elements)
 over $k$, which  implies that this element generates a separable 
 (resp.; these elements generate a radicial)  extension $U'\subset K=U'T'$ (resp.; $T'\subset K=U'T'$).

(2) For the same reason, $K\subset KU''$ is radicial, and of degree $p$, since $[U'':U']=p$. In particular, $KU''$ is 
 a radicial atom of $K\subset L$.
Moreover, assume that $k\subset L$ has FIP.  
Then, $KU''$ is the 
 unique radicial atom  of $[K,L]$. 
Deny and let $W\in[K,L],\ W\neq KU''$, be 
 a radicial atom of $[K,L]$.
 Then $K\subset U''W$ is a finite radicial extension which is not a chain, a contradiction. 
  It follows that   $KU''$ is the only radicial  atom of $[K,L]$.

(3) Let 
$V$ be a separable atom of $[K,L]$. 
We have the following diagram, 
$$\begin{matrix}
U'            &\to&       K     & \to&         V    &\to&       L      \\
\uparrow&{} & \uparrow & {} & \uparrow & {} & \uparrow \\
 k           &\to&       T'     & \to&   V\cap T &\to&       T
 \end{matrix}$$ 
 where $V\cap T$ is the separable closure of $T'\subset V$. In particular, $V\cap T\subseteq V$ is radicial, with $V\cap T\neq T'$ because $T'\subset V$ is not radicial, and $V\cap T\neq V$ because  $T'\subset V$ is not separable. We claim that $V$ is of the form $KT''$, where $T''$ is an atom of $T'\subset T$. Since $V\cap T\in]T',T]$, there is some $T''\in]T',V\cap T]$ such that $T'\subset T''$ is minimal. Then, $T'\subset T''\subseteq V\cap T$ implies $K=U'T'\subseteq U'T''\subseteq U'(V\cap T)\subseteq V$. But $K\subset V$  minimal implies that either $U'T''=K\ (*)$ or $U'T''=V\ (**)$.  In case $(*)$, we have $T''\in[T',K]$, with $T'\subset K$ radicial, so that $T'\subset T''$ is both radicial and separable, a contradiction. So, only case $(**)$ holds and $V=U'T''=U'(V\cap T)$
 $=U'T'T''=KT''$, because $T'\subset T''$.
 In particular, $T''=V\cap T$, as the separable closure of $T'\subset V$ because $T''\subseteq U'T''=V$ is radicial.
 
 Conversely,    an atom $T''$ of $[T',T]$ is such that $T'\subset T''$ is minimal separable.  The inclusion $T'\subset T''$ leads to $K=U'T'\subseteq U'T''$ separable, with $T''\subset U'T''$ radicial  and $K\neq U'T''$ by  a similar reasoning as before. In particular, $T''=U'T''\cap T$ as the separable closure of $T'\subseteq U'T''$. Assume that $K\subset U'T''$ is not minimal, so that there exists $V\in]K,U'T''[$, with $K\subset V$ separable. Then, $T'=K\cap T\subseteq V\cap T\subseteq U'T''\cap T=T''$. As above, $V\cap T\neq T'$, because $T'\subseteq V$ is not radicial, which leads to $V\cap T=T''$, and then to $U'T''=U'(V\cap T)\subseteq V$, a contradiction. Then, $K\subset U'T''$ is  minimal.
\end{proof} 

 \begin{definition}\label{except}  \cite{GHe} A finite field extension $k\subset L$ is said to be {\it exceptional} if $k=L_r$ and $L_s\neq L$. 
\end{definition}

 \begin{proposition} \label{9.3} Let $k\subset L$ be an FIP field extension with separable closure $T$,  radicial closure $U$ and $T,U\not\in\{ k,L\}$. Let $\{S_i\}_{i=0}^n$ (resp. $\{T_i\}_{i=0}^m,\ \{U_i\}_{i=0}^r,\ \{T_i\}_{i=m}^s$) be  the Loewy series of $k\subset L$ (resp. $k\subset T,\ k\subset U,\ T\subset L$). Then: 
\begin{enumerate}
 \item If $i\leq\inf(m,r)$, then $S_i=T_iU_i$.

 \item If $m\leq r$, then  $S_i=T_{i+m}$ for $i\geq m$. In particular, $\pounds [k,L]=n=s-m$.
 
 \item If $r<m$ and $U\subset L$ is separable, then 
  $U$ is 
   the   complement of $T$. Moreover, 
 $S_i=U_rT_i=UT_i$ for $i\in\{r,\ldots,m\}$, with   $L=S_n=T_s=UT_n$.
  In particular, 
  $\pounds [k,L]=n=m=\pounds [k,T]$ and $s=\pounds [k,T]+\pounds [k,U]=m+r$.
 
  \item Assume that $r<m$ and $U\subset L$ is not separable. For $i\geq r$, the $S_i$'s are gotten by induction on $i$ in the following way: if $S_i\subset L$ is exceptional, then $S_{i+1}=S_iT_{i+1}$;  if $S_i\subset L$ is not exceptional, then $S_{i+1}=V_iT_{i+1}$, where $V_i$ is the unique 
    radicial atom of $[S_i, L]$.
    In this way, we obtain the family $\{S_i\}_{i=r}^t$, for the least $t$ such that $S_t\in[T_m,L]$. As there exists some $l$ such that $S_t=T_l$,  then $S_i=T_{l+i-t}$, for $i\geq t$. In particular, $n=s+t-l$.
\end{enumerate}
 \end{proposition} 

\begin{proof} Set $p:=\mathrm{c}(k)$.

(1) We show by induction on $i$ that  $S_i=T_iU_i$ for any $i\leq\inf(m,r)$. For $i=0$, we have $k=S_0=T_0=U_0=T_0U_0$. Assume that for some   $i<\inf(m,r)$, we have $S_i=T_iU_i$. We are going to determine the atoms of $S_i\subset L$. Since any minimal field extension is either radicial or separable, it is enough to characterize any $V\in[S_i,L]$ 
  which is  either a radicial atom $(*)$, or a separable atom $(**)$.
   We use  Lemma \ref{9.1}. In case $(*),\ V=S_iU_{i+1}$, since 
     $ U_{i+1}$ is the only  radicial atom of  $[U_i,L]$.
     In case $(**),\  V=S_iT''$, where $T''$ is any atom of $T_i\subset T$. Because $S_{i+1}$ is the product of all atoms of $S_i\subset L$, we get that $S_{i+1}=S_iU_{i+1}T_{i+1}$. But, $S_i=T_iU_i$ with $T_i\subset T_{i+1}$ and $U_i\subset U_{i+1}$ leads to $S_{i+1}=U_{i+1}T_{i+1}$, and the induction is proved. 

(2) Assume that $m\leq r$. In view of (1), we have $S_m=T_mU_m=TU_m$ with   $S_m\subseteq L$  radicial. Indeed, $T_m\subset L$ is radicial, and then a chain, and $S_m\in[T_m,L]=\{T_i\}_{i=m}^s$. We have the following diagram
$$\begin{matrix}
  {} &       {}      & U_m & {}            &       {}        &         {} \\
 {} & \nearrow &   {}    & \searrow &       {}        &         {} \\
 k  &      {}       &   {}    & {}             & T_mU_m & \to & L \\
 {} & \searrow &   {}    & \nearrow &       {}        &         {} \\
 {} &      {}       & T_m & {}             &       {}        &         {}
\end{matrix}$$
with $k\subset U_m,\ T_m\subset T_mU_m$ both radicial, and $k\subset T_m,\ U_m\subset T_mU_m$ both separable. In particular, $[U_m:k]=[T_mU_m:T_m]=p^m$ since  
 $ U_{i+1}$ is a radicial atom of $[U_i,L]$
 for each $i\in\{0,\ldots,m-1\}$. It follows that $T_mU_m=S_m=T_{2m}$ and $S_i=T_{m+i}$ for any $i\in\{m,\ldots,n\}$, so that $m+n=s$
  because $S=S_n=T_s=T_{m+n}$.

(3) Assume that $m>r$, so that $S_r=T_rU_r=UT_r$.  Let $V$ be an atom of $S_r\subset L$. If 
 $V$ is a separable atom of $[S_r,L]$,
 then $V=S_rT''$, where $T''$ is an atom of $T_r\subset T$ in view of Lemma \ref{9.1}. In particular, $S_rT_{r+1}=U_rT_rT_{r+1}=U_rT_{r+1}$ is the product of the 
  separable atoms of $S_r\subset L$. 
  Since $U\subset L$ is separable,   
    there is no radicial atom in $[S_r,L]$.
    It follows that $S_{r+1}=U_rT_{r+1}$. An obvious induction shows that $S_i=U_rT_i$ for any $i\in\{ r,\ldots,m\}$. Since $TU\in[T,L]\cap [U,L]$, we get $L=TU$. 
     Therefore, $U$ is a complement of $T$ because $k=T\cap U$, which  is obviously unique.
     Moreover, $T=T_m\subset U_rT_m$ and $U=U_r\subset U_rT_m$, so that $L=U_rT_m=S_m=S_n=T_s=UT_n$
      since $U\subset L$ is separable and $T\subset L$ is radicial. Then, $n=m$.
 Moreover, $s=m+r$ since $[L:T]=[U:k]=p^r=p^{s-m}$.

(4) Assume that 
 $r<m$ and that 
 $U\subset L$ is not separable. We have $S_r=U_rT_r$ by (1). We get by induction on $i\geq r$ the $S_i$'s in the following way: Assume that $S_i$ is gotten. If $S_i\subset L$ is exceptional, there is no $V\in[S_i,L]$ such that $S_i\subset V$ is minimal radicial, then $S_{i+1}=S_iT_{i+1}$ as in  case (3). If $S_i\subset L$ is not exceptional, there is a   unique 
  radicial atom $V_i$ of  $[S_i,L]$. 
  But in this case, since $S_iT_{i+1}$ is the product of 
   separable 
  atoms of $[S_i,L]$, then $S_{i+1}=S_iT_{i+1}V_i=V_iT_{i+1}$ because $ S_i\subset V_i$. Since $\{S_i\}_{i=r}^n$ is an increasing sequence, there is a  least $t\geq r$ such that $S_t\in[T_m,L]$. Indeed, $T_m\subseteq S_m$, and there exists some $l\geq m$ such that $S_t=T_l$ because $[T_m,L]=\{T_i\}_{i=m}^s$. Then $S_i=T_{l+i-t}$, for $i\geq t$, 
  which implies that $s=l+n-t$, that is $n=s+t-l$. 
\end{proof} 

  \begin{proposition}\label{9.31} Let $k\subset L$ be an  FCP radicial field extension. Set $p:=\mathrm{c}(k)$ and $[L:k]=p^n$.
  \begin{enumerate}
\item Then $\mathcal S[k,L]=\{x\in L\mid x^p\in k\}$.  
  
  \item If $k\subset L$ has FIP, the Loewy series of $k\subset L$ is $[k, L]$ and $\pounds [k,L]=n$.
   \end{enumerate} 
 \end{proposition}
  
  \begin{proof}  (1) $\mathcal S[k,L]=\prod_{A\in\mathcal A}A$, where $\mathcal A$ is the set of atoms of $[k,L]$. Now, $A\in\mathcal A\Leftrightarrow k\subset A$ is minimal $\Leftrightarrow A=k[x]$ with $k\subset k[x]$  minimal radicial $\Leftrightarrow [k[x]:k]=p\Leftrightarrow A=k[x]$ with $x^p\in k$. In particular, $t^p\in k$ for any $t\in A$  since $\mathrm{c}(k)=p$. Let $y\in\mathcal S[k,L]$. Then $y$ is a finite sum of products $z:=x_1\cdots x_n$ of elements  of atoms of $[k,L]$. But $z^p=x_1^p\cdots x_n^p\in k$, which yields that $y^p\in k$. Then, $\mathcal S[k,L]\subseteq \{x\in L\mid x^p\in k\}$. Conversely, if $x^p\in k$ for some $x\in L\setminus k$, it follows that $k\subset k[x]$ is minimal, so that $k[x]\in\mathcal A$ which leads to $x\in \mathcal S[k,L]$. Then, $\mathcal S[k,L]= \{x\in L\mid x^p\in k\}$.
  
  (2) Obvious since $[k, L]$ is a chain.
  \end{proof} 
  
   \begin{remark}\label{9.32} Contrary to FCP separable field extensions which are always FIP, there exist FCP radicial field extensions which are not FIP. Take for instance $k:={\mathbb Z}/2{\mathbb Z}(Y,T)$, the field of rational functions over ${\mathbb Z}/2{\mathbb Z}$ in two indeterminates $Y$ and $T$. Let $y$ (resp. $t$) be the class of $Y$ (resp. $T$) in $k$ and let $\alpha$ (resp. $\beta$) be a zero of $F(X):=X^2-y$ (resp. $G(X):=X^2-t$). Then, $k\subset k[\alpha,\beta]$ is an FCP radicial extension of length 2. But, it has not FIP by \cite[Theorem 6.1 (8) (a)]{Pic 6}, because $|[k,k[\alpha,\beta]]|\geq 4$ gives $|[k,k[\alpha,\beta]]|=\infty$. 
 \end{remark} 
 
 Actually, the following proposition gives a characterization of FIP radicial field extensions.
 
  \begin{proposition} \label{9.33}  Let $k\subset L$ be an FCP radicial  field extension. The following conditions are equivalent:
 \begin{enumerate}

\item $k\subset L$ is a chain.

\item $k\subset L$ is distributive.

\item $k\subset L$ has FIP. 
 \end{enumerate}
\end{proposition}

\begin{proof} (1) $\Rightarrow$ (2) by \cite[Proposition 2.3]{Pic 10}.

(1) $\Rightarrow$ (3) since $k\subset L$ has  FCP.

(3) $\Rightarrow$ (1) by \cite[Lemma 4.1]{Pic 10}.

(2) $\Rightarrow$ (1) Assume that $k\subset L$ is distributive and not a chain. There exist $K_1,K_2,K_3\in[k,L]$ such that $K_1\subset K_i$ is minimal for $i=2,3$ with $K_2\neq K_3$. We get that $|[K_1,K_2K_3]|=4$ because $\ell[K_1,K_2K_3]|=2$. Indeed, $k\subset L$ is distributive, so that $[K_1,K_2K_3]$ does not contain a diamond \cite[Theorem 1, page 59]{G}, a contradiction with  \cite[Theorem 6.1 (8) (a)]{Pic 6}.
\end{proof}

 The last case to consider is the case of a finite separable field extension. We recall here some results gotten in \cite{Pic 10}.

Let $L:=k[x]$ be a finite separable (whence FIP) field extension of $k$   and $f(X)\in k_u[X]$ (the set of monic polynomials of $k[X]$) 
 {\it the} 
 minimal polynomial of $x$ over $k$. 
If $g(X)\in L_u[X]$ divides $f(X)$, we denote by $K_g$ the $k$-subalgebra of $L$ generated by the coefficients of $g$.  For any $K\in[k,L]$, we denote by $f_K(X)\in K_u[X]$ the minimal polynomial of $x$ over $K$. The proof of the Primitive Element Theorem shows that $K=K_{f_K}$.
Of course, $f_K(X)$ divides $f (X)$ in $K[X]$ (and in $L[X]$). We set $\mathcal D:=\{f_K\mid K\in[k,L]\}$. Then, $(\mathcal D,\leq)$ is a poset for the order $\leq$ defined as follows: if $f_K,f_{K'}\in\mathcal D$, then $f_K\leq f_{K'}$ if and only if $f_K$ divides $f_{K'}$ in $L[X]$, which is equivalent to $ K'\subseteq K$ by \cite[Lemma 4.7]{Pic 10}. In particular,  $\inf$ is  {\rm gcd} in $\mathcal D$.   

\begin{corollary}\label{4.251} \cite[Corollary 4.9]{Pic 10} The map  $\varphi :[k,L]\to\mathcal D$ defined by $K\mapsto f_K$ is a reversing order bijection such that  $f_{KK'}=\inf(f_K,f_{K'})$ for $K,K'\in[k,L]$.
\end{corollary} 

 \begin{proposition}\label{9.4} Let $k\subset L:=k[x]$ be a finite separable field extension of $k$ and $K\subset K'$ a subextension. Then, $K\subset K'$ is minimal if and only if $f_{K'}$ is a maximal proper divisor of $f_K$ in $\mathcal D$.
\end{proposition} 

\begin{proof} In view of Corollary \ref{4.251}, we have $K\subset K'$ if and only if $f_{K'}$ divides $f_K$ in $\mathcal D$. Moreover, $K\subset K'$ is minimal if and only if there is no $K''\in[k,L]$ such that $K\subset K''\subset K'$  if and only if there is no proper divisor of $f_K$ divided strictly by $f_{K'}$ if and only if $f_{K'}$ is a maximal proper divisor of $f_K$ in $\mathcal D$.
\end{proof}

 \begin{proposition}\label{9.5} Let  $k\subset L:=k[x]$ be a finite separable field extension of $k$. The Loewy series $\{S_i\}_{i=0}^n$ of $k\subset L$ is gotten by induction in the following way: $S_0=k$ and for $i\in\{0,\ldots,n-1\}$, we have $S_{i+1}=K_g$, where $g=\inf\{h\in\mathcal D\mid h$ is a maximal proper divisor of $f_{S_i}$ in $\mathcal D\}$.
\end{proposition} 

\begin{proof} For a given $S_i$, we have $S_{i+1}=\prod \{V\mid S_i\subset V$ minimal$\}$. In view of Proposition \ref{9.4}, $S_i\subset V$ is minimal $\Leftrightarrow f_V$ is a maximal proper divisor of $f_{S_i}$. It follows from Corollary \ref{4.251} that $S_{i+1}=K_g\Leftrightarrow g=f_{S_{i+1}}=\inf\{h\in\mathcal D\mid h$ is a maximal proper divisor of $f_{S_i}\}$.
\end{proof} 

 \begin{proposition}\label{9.6} Let $k\subset L:=k[x]$ be a minimal separable field extension. Let $N$ be the normal closure of $L$. Then, $\pounds [k,N]=1$. Moreover, the following conditions are equivalent:
\begin{enumerate}
 \item $k\subset N$ is distributive.
  \item $k\subset N$ is Boolean.
   \item $k\subset N$ is a cyclic extension and $[N:k]$ is square-free.
   \end{enumerate}
\end{proposition} 

\begin{proof} Let $G$ be the Galois group of $k\subset N$. Then $N=k[\{\sigma(x)\mid\sigma\in G\}]=\prod_{\sigma\in G}k[\sigma(x)]$. Since  $k\subset k[\sigma(x)]$  is obviously minimal for any $\sigma\in G$,  it follows that $N=\mathcal S[k,N]$, so that $\pounds [k,N]=1$.

(1) $\Leftrightarrow$ (2) by Lemma \ref{8.2} and by definition of a Boolean extension.

(2) $\Leftrightarrow$ (3) by \cite[Theorem 4.19]{Pic 10}.
\end{proof} 

 \begin{corollary}\label{9.7} Let $k\subset L$ be a finite Galois extension with Galois group $G$. The socle   $\mathcal S[k,L]$ is globally invariant by the elements of $G$. Moreover, $k\subset \mathcal S[k,L]$ is Galois.
\end{corollary} 

\begin{proof} By definition, $\mathcal S[k,L]=\prod_{A\in\mathcal A}A$. For any $A\in\mathcal A$, there exists $x_A\in A$ such that $A=k[x_A]$, so that $\mathcal S[k,L]=\prod_{A\in\mathcal A}k[x_A]$. For any $\sigma\in G$, we have $\sigma(\mathcal S[k,L])=\sigma(\prod_{A\in\mathcal A}k[x_A])=\prod_{A\in\mathcal A}k[\sigma(x_A)]$. Obviously, $k[\sigma(x_A)]\in\mathcal A$ which yields $\sigma(\mathcal S[k,L])\subseteq \mathcal S[k,L]$. But $\sigma$ being a $k$-isomorphism, for any $x\in L$ such that $k[x]\in\mathcal A$, and setting $y:=\sigma^{-1}(x)$, that is $x=\sigma(y)$, we have $k[x]=k[\sigma(y)]$, with $k[y]\in\mathcal A$. To conclude, $\mathcal S[k,L]\subseteq \sigma(\mathcal S[k,L])$ and $ \mathcal S[k,L]=\sigma(\mathcal S[k,L])$. This equality shows that $k\subset \mathcal S[k,L]$ is Galois by \cite[Proposition 5, page A V.54]{Bki A}.
\end{proof}

We  give here a complete study of
  the case of finite Galois distributive field extensions.  They are evidently  FIP. Recall that a finite Galois field extension is minimal if and only if its degree is a prime integer \cite[Proposition 2.2]{Pic}
  and is Boolean if and only if it is a cyclic extension with a square-free degree   \cite[Theorem 4.19]{Pic 10}. 

\begin{proposition} \label{8.121} A finite Galois field extension is distributive if and only if it is cyclic. 
\end{proposition} 

\begin{proof}  Let $k\subset L$ be a finite Galois field extension with Galois group $G$ and let $\mathcal G$ be the set of subgroups of $G$. In view of the Fundamental Theorem of Galois Theory, the maps $\varphi:[k,L]\to\mathcal G$ defined by $\varphi (K):=$Aut$_K(L)$, the group of $K$-automorphisms of $L$, for each $K\in[k,L]$ and $\psi:\mathcal G\to[k,L]$ defined  by $\psi(H):=$Fix$(H)$, the fixed field of $H$ in $L$, for each $H\in\mathcal G$ are reversing order isomorphisms of lattices, with $\varphi=\psi^{-1}$ (see \cite[Corollaire 2, page A V.65]{Bki A}). It follows that $k\subset L$ is distributive if and only if $\mathcal G$ is distributive if and only if $G$ is cyclic \cite[Theorem 19.2.1]{H}. 
\end{proof}

\begin{theorem} \label{8.131} Let $k\subset L$ be a finite cyclic field extension. Set $r:=[L:k]=\prod_{i=1}^mp_i^{\alpha_i}$, where the $p_i$ are distinct prime integers. Let 
$\mathbb D_r$
 be  the set of divisors of $r$  and  $\mathcal A$ be the set of atoms of $[k,L]$. Then:
\begin{enumerate}
 \item For each $i\in\mathbb N_m$, there is a unique $A_i\in\mathcal A$ such that $[A_i:k]=p_i$. Let $T\in[k,L]$. Then, $T\in\mathcal A$ if and only if $[T:k]=p_i$ for some $i\in\mathbb N_m$. 
 
\item $\mathcal S[k,L]=\prod_{i=1}^mA_i$.
 
\item Set $\alpha:=\sup\{\alpha_i\mid i\in\mathbb N_m\}$. Let $\{S_j\}_{j=0}^n$ be the Loewy series of $k\subset L$ and let $j\in\{0,\ldots,n\}$. Then, $[S_j:k]=\prod_{i=1}^mp_i^{\beta_i}$, where either $\beta_i=\alpha_i$ if $\alpha_i<j$ or $\beta_i=j$ if $\alpha_i\geq j$. Moreover, $T\in[k,L]$ is an atom of $[S_j,S_{j+1}]$ if and only if $[T:k]=\prod_{i=1}^mp_i^{\gamma_i}$, where there exists a unique $i_0$ such that $\beta_{i_0}<\alpha_{i_0}$,  satisfying
  $\gamma_i=\beta_i$ for each $i\neq i_0$ and $\gamma_{i_0}=\beta_{i_0}+1$. In particular, $\alpha=n$. 
  
\item   $S_j\subset S_{j+1}$ is Boolean for each $j\in\{0,\ldots,n-1\}$.
 
\item Let $T\in[k,L]$. Then, $T$ is a $\Pi$-irreducible element of $[k,L]$  if and only if $[T:k]=p_i^{\beta_i}$ for some $\beta_i\in\mathbb N_{\alpha_i}$ 
 and some $i\in\mathbb N_m$. 

\item $|[k,L]|=\mathrm{d}([L:k])=\prod_{i=1}^m(\alpha_i+1)=
|\mathbb D_r|$. 
Let ${\mathcal A}_j$ be the set of atoms of $[S_j,S_{j+1}]$. Then, $\ell[k,L]=\sum_{j=0}^{\alpha-1}|{\mathcal A}_j|=\sum_{i=1}^m\alpha_i$ which is  the number of  $\Pi$-irreducible elements of $[k,L]$. 

\item The following conditions are equivalent:
\begin{enumerate}
 \item $k\subset L$ is a $\mathcal P$-extension.
\item  $k\subset L$ is either Boolean or a chain.
  \item   either $m=1$ or $\alpha=1$.
\end{enumerate}

\end{enumerate}
\end{theorem} 

\begin{proof} First observe that any subextension of $k\subset L$ is still cyclic 
 and $k\subset L$ is Galois distributive. Moreover, there is a lattice  isomorphism 
$\mathbb D_r\to [k,L]$,
 where $d\mapsto T$ such that $[T:k]=d$.

(1) Let $G$ be the Galois group of $k\subset L$. Then, $G\cong {\mathbb Z}/r{\mathbb Z}$, with $|G|=r$. Since $k\subset L$ is cyclic, it follows that for any integer 
$d\in \mathbb D_r$,
 there exists a unique $T\in[k,L]$ such that $[T:k]=d$, and conversely, 
 $[T:k]\in\mathbb D_r$
   for any $T\in[k,L]$. In particular, $k\subset T$ is minimal if and only if $[T:k]$ is a prime integer. Then, $T\in\mathcal A$ if and only if $[T:k]=p_i$ for some $i\in\mathbb{N}_m$, and for each $i\in\mathbb{N}_m$, there is a unique $A_i\in\mathcal A$ such that $[A_i:k]=p_i$.

(2) By definition, $\mathcal S[k,L]=\prod_{A\in\mathcal A}A=\prod_{i=1}^mA_i$.

(3) We  show by induction on $j\geq 0$ that $[S_j:k]=\prod_{i=1}^mp_i^{\beta_i}$, where either $\beta_i=\alpha_i$ if $\alpha_i<j$ or $\beta_i=j$ if $\alpha_i\geq j$.
 The induction hypothesis holds 
  clearly for $j=0$.

Assume that the induction hypothesis holds for some 
 $j\in\{0,\ldots,n-1\}$ so that $S_j\subset L$. Then, $[S_j:k]=\prod_{i=1}^mp_i^{\beta_i}$, where either $\beta_i=\alpha_i$ if $\alpha_i<j$ or $\beta_i=j$ if $\alpha_i\geq j$. Set $m_j:=|\{i\in\mathbb{N}_m\mid \alpha_i> j\}|$. There is no harm to renumber the $\alpha_i$'s so that $\alpha_i> j$ for each $i\leq m_j$ and $\alpha_i\leq  j$ for each $i> m_j$.  Let  $T$ be an atom of $[S_j,S_{j+1}]$, so that $S_j\subset T$ is minimal and $[T:S_j]=p_{i_0}$, for some $i_0\in \mathbb{N}_m$.  In particular, $[T:k]=p_{i_0}^{\beta_{i_0}+1}\prod_{i\neq i_0}p_i^{\beta_i}$, so that $\beta_{i_0}+1\leq \alpha_{i_0}$, which leads to $\beta_{i_0}<\alpha_{i_0}$. Then, $\beta_{i_0}=j$, that is $\alpha_{i_0}>j$ and $i_0\leq m_j$.

Conversely, if $i\leq m_j$, then $\alpha_i>j$, so that there exists $T\in[S_j,L]$ such that $[T:S_j]=p_i$, and $T$ is an atom of $S_j\subset L$. 

Since $S_{j+1}$ is the product of all atoms of $S_j\subset L$, it follows that 
  $[S_{j+1}:S_j]=\prod_{i=1}^{m_j}p_i$, giving $[S_{j+1}:k]=[S_{j+1}:S_j][S_j:k]=(\prod_{i=1}^{m_j}p_i)(\prod_{i=m_j+1}^mp_i^{\beta_i})(\prod_{i=1}^{m_j}p_i^{\beta_i})=\prod_{i=1}^mp_i^{{\beta}'_i}$, where ${\beta}'_i=\beta_i +1$ if $i\leq m_j$ and ${\beta}'_i=\beta_i$ for $i>m_j$. This means the following: if $\alpha_i\geq j+1>j$, then, ${\beta}'_i=\beta_i+1=j+1$,  if $\alpha_i< j$, then, ${\beta}'_i=\beta_i=\alpha_i$, and, if $\alpha_i=j$, then, ${\beta}'_i={\beta}_i=j=\alpha_i$. Hence, the  induction hypothesis  holds for $S_{j+1}$. 
In particular, $L=S_{\alpha}=S_n$ and $\alpha=n$. 

 (4)  $S_j\subset S_{j+1}$ is Boolean for each $j\in\{0,\ldots,n-1\}$ by Proposition \ref{8.3} since $k\subset L$ is distributive. In particular, we recover the fact that $[S_{j+1}:S_j]$ is square-free for each $j\in\{0,\ldots,n-1\}$.

(5) Let $T\in[k,L]$ be such that  $[T:k]=p_i^{\beta_i}$ for some $\beta_i\in \mathbb{N}_{\alpha_i}$ 
 and some $i\in\mathbb N_m$. 
 It follows that $[k,T]$ is a chain, since so is the Galois group of $k\subset T$ (isomorphic to $\mathbb Z/ p_i^{\beta_i}\mathbb Z$). Then, $T$ is $\Pi$-irreducible by Proposition  \ref{5.51}.

Assume now that $[T:k]$ is divided by at least two distinct prime integers. After a suitable reordering, we may assume that $[T:k]=\prod_{i=1}^rp_i^{{\beta}_i},\ r>1$ and $\beta_i>0$ for each $i\in \mathbb{N}_r$. In view of the Fundamental Theorem of Galois Theory, there exist $T_1,T_2\in[k,T]$ such that $[T:T_i]=p_i$ for $i=1,2$, so that $T_1\subset T$ and $T_2\subset T$ are two minimal field extensions as it is recalled before Proposition \ref{8.121}. Then,  $T$ is not  $\Pi$-irreducible by Proposition  \ref{5.51}.

(6) The equality $|[k,L]|=\mathrm{d}([L:k])$ is obvious because of the lattice isomorphism recalled at the beginning of the proof. By \cite[Lemma 4, p.486]{BM}, $\ell[k,L]$ is the number of $\Pi$-irreducible elements of $[k,L]$. Then, $\ell[k,L]=\sum_{i=1}^m\alpha_i$ by (5). To end, $\ell[k,L]=\sum_{j=0}^{\alpha-1}\ell[S_j,S_{j+1}]$ 
 by Proposition \ref{8.3}.
But each $[S_j,S_{j+1}]$ is Boolean, so that $\ell[S_j,S_{j+1}]=|{\mathcal A}_j|$ 
by 
 \cite[Theorem 3.1]{Pic 10}. This yields $\ell[k,L]=\sum_{j=0}^{\alpha-1}|{\mathcal A}_j|$. 

We can remark that we may consider $\ell[k,L]$ in two different ways :  when we write 
$$\begin{matrix}
 \ell[k,L]= & 1         & + & \ldots & +   & \ldots &  =\alpha_1  \\
 {}           & 1           & + & \ldots & +   & \ldots &   =\alpha_2  \\
{}  & \cdots & \cdots & \cdots &\cdots & \cdots & \cdots  \\
 {}          & 1           & +      & \ldots & + & \ldots &   =\alpha_m  \\
  {}  & = & {} & = & {} & {}   & {} \\
{}  & |{\mathcal A}_0| & {} & |{\mathcal A}_j| & {} & \cdots   & {}
\end{matrix}$$
in line $i$, we have the power of $p_i$ in $[L:k]$, which is the number of $\Pi$-irreducible elements whose degree of extension over $k$ is a power of $p_i$, and in column $j$, we have the number of atoms of $S_j\subset S_{j+1}$, with either $1$ or $0$ instead of $\ldots$, and $0$ after and under each $0$.

(7) We  discuss with respect to $m$ and $\alpha$.

If $m=1$, then $k\subset L$ is a chain and 
 a $\mathcal P$-extension  
 by Corollary ~\ref{8.31}. 

Assume $m>1$.

If $\alpha=1$, then $[L:k]=\prod_{i=1}^mp_i$ shows that $k\subset L$ is Boolean by 
 the remark before Proposition  \ref{8.121},  
 because its degree is square-free,
 and then 
  a $\mathcal P$-extension  
  since $[k,L]=[k,S_1]$.
  
 Assume that $\alpha>1$. After reordering, we may assume that $\alpha_1>1$. There exists $T\in[k,L]$ such that 
 $[T:k]=p_1^2$.
 Since $[S_1:k]=\prod_{i=1}^mp_i$ by (3), we get that  
  $T\not\in[k,S_1]\cup[S_1;L]$, so that $[k,L]\neq\cup_{i=0}^{n-1} [S_i, S_{i+1}]$ and
  $k\subset L$ is not a $\mathcal P$-extension 
  in this case. 
  In particular, $k\subset L$ is neither Boolean nor a chain. 

Gathering the different cases we get (7).
\end{proof} 

 \begin{remark} \label{8.132} Comparing  Theorem \ref{8.5} and Theorem \ref{8.131} (4), we see that the latter is a generalization of  Theorem \ref{8.5} in case of a finite cyclic field extension. Indeed, $k\subset L$ is distributive and  $S_j\subset S_{j+1}$ is Boolean for each $j\in\{0,\ldots,n-1\}$ whatever $k\subset L$ is  a $\mathcal P$-extension or not.
\end{remark}

\end{document}